\documentclass[12pt]{article}
\usepackage{geometry,url,latexsym,bm,amsthm,amssymb,amsfonts,amsbsy,amsmath,graphics,graphicx,color}
\usepackage[english]{babel}

\newtheorem{rem}{Remark}
\newtheorem{theo}{Theorem}

\def\pmatrix{\left(\begin{array}}
\def\endpmatrix{\end{array}\right)}
\def\RR{{\mathbb{R}}}

\title{Mathematical aspects relative to the fluid statics of a self-gravitating perfect-gas isothermal sphere}
\author{Pierluigi Amodio\thanks{Dipartimento di Matematica, Universit\`a di Bari, Italy,  {\tt pierluigi.amodio@uniba.it}, {\tt felice.iavernaro@uniba.it}, {\tt arcangelo.labianca@uniba.it}, {\tt monica.lazzo@uniba.it}, {\tt lorenzo.pisani@uniba.it}},
Domenico Giordano\thanks{European Space Agency, ESTEC, Aerothermodynamics Section (retired), The Netherlands, {\tt dg.esa.retired@gmail.com}},
Felice Iavernaro$^*$,\\
Arcangelo Labianca$^*$, Monica Lazzo$^*$,\\
Francesca Mazzia\thanks{Dipartimento di Informatica, Universit\`a di Bari, Italy, {\tt francesca.mazzia@uniba.it}} and Lorenzo Pisani$^*$
}

\date{}

\begin{document}
\maketitle

\begin{abstract}
In the present paper we analyze and discuss some mathematical aspects of the fluid-static configurations of a self-gravitating perfect gas enclosed in a spherical solid shell. 
The mathematical model we consider is based on the well-known Lane--Emden equation, albeit under boundary conditions that differ from those usually assumed in the astrophysical literature. 
The existence of multiple solutions requires particular attention in devising appropriate numerical schemes apt to deal with and catch the solution multiplicity as efficiently and accurately as possible.
In sequence, we describe some analytical properties of the model, the two algorithms used to obtain numerical solutions, and the numerical results for two selected cases. 

\end{abstract}

\textbf{keyword} 
Self-gravitating gas, Lane--Emden equation, Multiple solutions.



\section{Introduction}
The motivation behind our research in gravitational fluid dynamics is described in details in the introduction to our recent paper \cite{ESA19a}. The core purpose of that paper was to study the static configurations of a self-gravitating isothermal sphere (see Figure \ref{sphere}).
\begin{figure}[h]
\begin{center}
	\includegraphics[keepaspectratio=true, trim= 5ex 8ex 4ex 10ex , clip , width=.75\columnwidth]{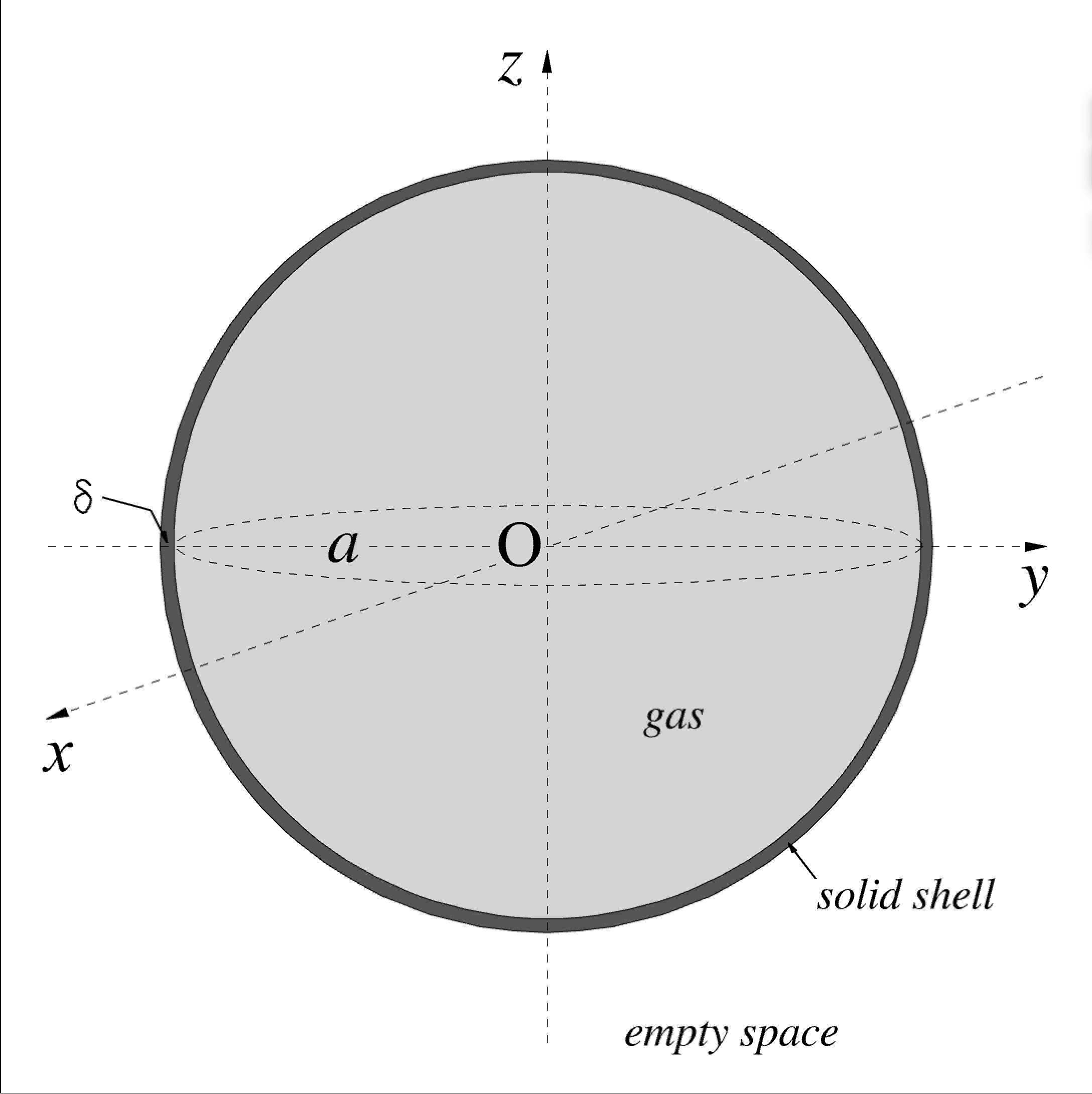}
\end{center}
	\caption{Study case: fluid statics of a self-gravitating isothermal perfect gas inside a spherical solid shell with internal radius $a$ and thickness $\delta$.\hfill\ }\label{sphere}
\end{figure}
The emphasis therein was mainly on physics, fluid statics and thermodynamics in particular; we described thoroughly the adopted physical model and the ensuing governing equations, and discussed numerical results.
The spherically symmetric mathematical problem inside the gas turns out to hinge completely on the solution of the isothermal Lane-Emden equation \cite{emden}
\begin{equation}\label{rho_c0}
\dfrac{d^2 \log \rho}{d r^2} +\dfrac{2}{r} \dfrac{d \log \rho}{d r} +\dfrac{4 \pi G}{RT} \rho =0,
\end{equation}
which provides the mass-density radial profile. Once this is available, pressure and gravitational field follow straightforwardly from
\begin{equation}\label{p}
   p = \rho R T
\end{equation}
and 
\begin{equation}\label{gf}
  g = R T \dfrac{d\log \rho}{d r}.
\end{equation}
The following notation is used in equations (\ref{rho_c0})-(\ref{gf}):
\begin{tabbing}
	i \= xx \= xxxxxxxxxx \kill
	\> $r$             \> radial distance from origin \\
	\> $\rho$       \> mass density \\
	\> $G$            \> gravitational constant, $6.67428\cdot10^{-11}$m$^{3}\cdot$kg$^{-1}\cdot$s$^{-2}$ \\
	\> $T$            \> temperature \\
	\> $R$            \> gas constant \\ 
	\> $p$          \> pressure \\
	\> $g$          \> gravitational field
\end{tabbing}
Equation~\eqref{rho_c0} is complemented with two conditions.
The condition
\begin{equation}\label{left_condition}
\left. \dfrac{d \log \rho}{d r} \right|_{r=0}=0
\end{equation}
is necessary from a mathematical point of view, to guarantee regularity of solutions (see \cite{jo-lun}), and is consistent with Gauss' theorem, requiring the vanishing of the gravitational field at the center of the gas sphere.
The boundary condition
\begin{equation}\label{right_condition}
\left. \dfrac{d \log \rho}{d r}\right|_{r=a} = -\dfrac{G m_g}{a^2 R T}
\end{equation}
reflects the gravitational-field value established at the internal wall of the solid shell by Gauss' theorem and  the spherical symmetry of the physical system.
In imposing the  boundary condition \eqref{right_condition}, we diverge from the standard followed in the astrophysical literature, which adopts the prescription of the gas density [$\rho(0)$] at the center of the  sphere.
Ample discussion of the physical arguments that motivate and justify our choice is provided in \cite{ESA19a}.

According to standard practice, we reformulate the mathematical problem [\eqref{rho_c0}; \eqref{left_condition} and \eqref{right_condition}] by introducing nondimensional radial coordinate and density together with corresponding scale factors marked with a tilde 
\begin{equation}\label{scal}
     r = \tilde r \, \eta \qquad \mbox{and} \qquad \rho(r)=\tilde \rho \, \xi(\eta),
\end{equation}
so that the governing differential equation \eqref{rho_c0} and its associated boundary conditions \eqref{left_condition} and \eqref{right_condition} turn into the nondimensional form
\begin{equation*}
    \dfrac{d^2 \log \xi}{d \eta^2} +\dfrac{2}{\eta} \dfrac{d \log \xi }{d \eta} +\dfrac{4 \pi G \tilde r^2 \tilde \rho}{RT} \xi = 0
\end{equation*}
and
\begin{equation*}
     \left.\dfrac{d \log\xi}{d \eta}\right|_{\eta=0} = 0,  \qquad \left. \dfrac{d \log \xi}{d \eta}\right|_{\eta= a/\tilde r} = -\dfrac{G m_g \tilde r}{a^2 R T}.
\end{equation*}
One further step consists in introducing the function $y(\eta)=\log \xi(\eta)$ and  the characteristic numbers
\begin{equation}\label{cn}
    \Pi_1 = \dfrac{4 \pi G \tilde r^2 \tilde \rho}{RT}, \qquad  \Pi_4=\dfrac{a}{\tilde r}, \qquad N = \dfrac{Gm_g}{a R T} .
\end{equation}
which control the mathematical problem. For details on the conceptual implications of  turning equations such as \eqref{rho_c0} into nondimensional form, we refer the reader to \cite[Section 2.2]{ESA19a}. After that step, the mathematical problem achieves the final form
\begin{subequations}\label{y_d}
\begin{alignat}{1}
& \dfrac{d^2 y}{d \eta^2} +\dfrac{2}{\eta} \dfrac{d y }{d \eta} +\Pi_1 e^y = 0,  \label{y_d_equat} \\[1mm]
& \left.\dfrac{d y}{d \eta}\right|_{\eta=0} =0, \qquad \left. \dfrac{d y}{d \eta}\right|_{\eta=\Pi_4} = -\dfrac{N}{\Pi_4}, \label{y_d_BC}
\end{alignat}
\end{subequations}
which constitutes the starting point of the present analysis.

In this paper, we concentrate on the mathematical aspects related to the governing differential equation \eqref{y_d_equat}, its boundary conditions \eqref{y_d_BC} and its solutions.
In \mbox{Section \ref{sec:3}}, we carry out a theoretical analysis of \eqref{y_d} and, among other things, we show that multiple solutions are possible. 
In \mbox{Section \ref{sec:4}}, we introduce two different discretizations to solve the density equation.
In \mbox{Section \ref{sec:5}}, we illustrate and discuss numerical tests to highlight the behavior of the solutions. 
Some concluding remarks are drawn in Section \ref{sec:6}. In the Appendix, multiplicity results for higher dimensional problems are obtained by combining the analysis of Section  \ref{sec:3} with numerical evidence.

\section{Some analytical results}
\label{sec:3}
In this section we discuss the existence and multiplicity of solutions to problem~\eqref{y_d}.
Our main tool is the following representation theorem, inspired by techniques developed in the study of \eqref{y_d_equat} under Dirichlet boundary conditions (for instance, see \cite{gel}).

\begin{theo}\label{rep}
A function $y$ solves problem~\eqref{y_d} if, and only if,
\begin{equation}\label{repr}
y(\eta)=U\left(\dfrac{\sigma}{\Pi_4} \, \eta \right)  + \ln\left(\dfrac{\sigma^2}{\Pi_1 \, \Pi_4^2 }\right) ,    \qquad \eta \in [0,\Pi_4],
\end{equation}
where $U$ is the unique maximal forward solution of the initial value problem
\begin{equation}\label{IVPprob}
\left\{
\begin{array}{l}
u''+\dfrac{2}{t} \, u' + e^{u}=0,  \\[2mm]
u'(0)=u(0)=0,
\end{array}\right.
\end{equation}
and $\sigma \in (0,\infty)$ satisfies
\begin{equation}\label{num-eq}
{\sigma} \, U'(\sigma)  = - N  .
\end{equation}
\end{theo}
\begin{proof}
To begin with, note that \eqref{IVPprob} has a unique maximal forward solution~$U$,
despite being singular at $t=0$.  From $(t^{2} \, U'(t))' = - t^{2} \, e^{U(t)} < 0$ and $U'(0)=0$, it follows that $U$ is decreasing in its interval of existence; from this and $U(0)=0$, it follows that $U$ is negative throughout. This implies $U(t) \ge - {t^2}/{6}$ throughout, whence $U$ is global.

Direct calculations show that a function $y$, defined as in~\eqref{repr}, solves~\eqref{y_d}. To prove the ``only if'' part of the statement, suppose that a function $y$ solves~\eqref{y_d}.
Define
\begin{equation*}
z(t)=y\left({e^{-y(0)/2}}\,{\Pi_1^{-1/2} } \, t\right)-y(0) ,
\end{equation*}
and observe that $z$ solves~\eqref{IVPprob} in $[0,e^{y(0)/2} \, \Pi_1^{1/2}\, \Pi_4]$. Uniqueness implies $z=U$, whence
\begin{equation}\label{what}
y(\eta) = U \left(e^{y(0)/2} \, \Pi_1^{1/2} \, \eta\right) + y(0) .
\end{equation}
Since, by assumption, $y'(\Pi_4) = - N/\Pi_4$, from~\eqref{what} we get
\begin{equation*}
e^{y(0)/2} \, \Pi_1^{1/2} \, \Pi_4 \, U' \left(e^{y(0)/2} \, \Pi_1^{1/2} \, \Pi_4\right) = -N,
\end{equation*}
that is, \eqref{num-eq} is satisfied with $\sigma=e^{y(0)/2} \, \Pi_1^{1/2} \, \Pi_4$.
Solving for $y(0)$ and substituting into~\eqref{what} we get that $y$ satisfies~\eqref{repr}. \qed
\end{proof}

\bigskip
Note that, by Theorem~\ref{rep}, the profile of every solution to~\eqref{y_d} is obtained, via scalings and translations, from the profile of $U$ (see Figure~\ref{fig:U-W}, left).


\begin{figure}
\centerline{
\includegraphics[width=.56\textwidth]{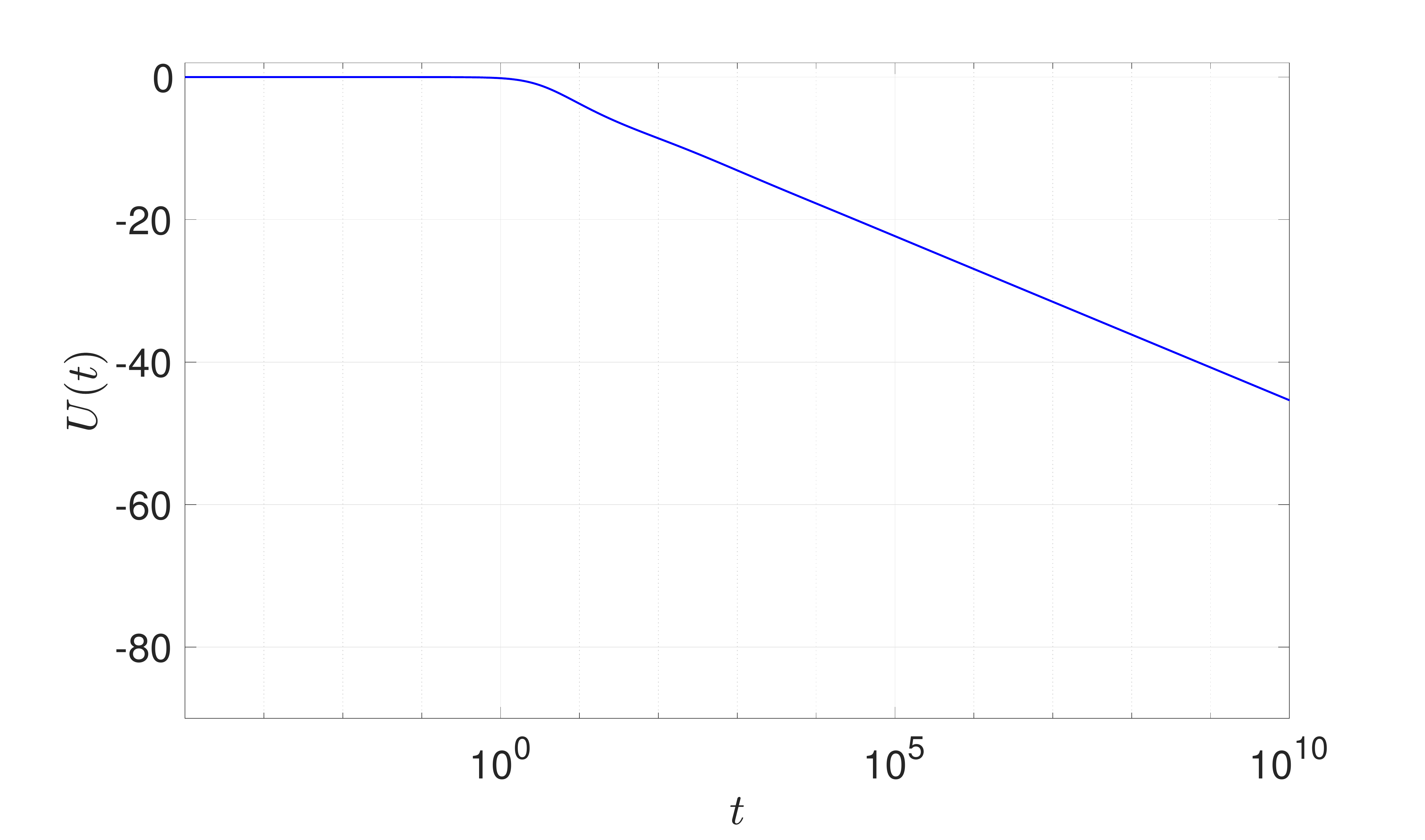} \hspace*{-.8cm}
\includegraphics[width=.56\textwidth]{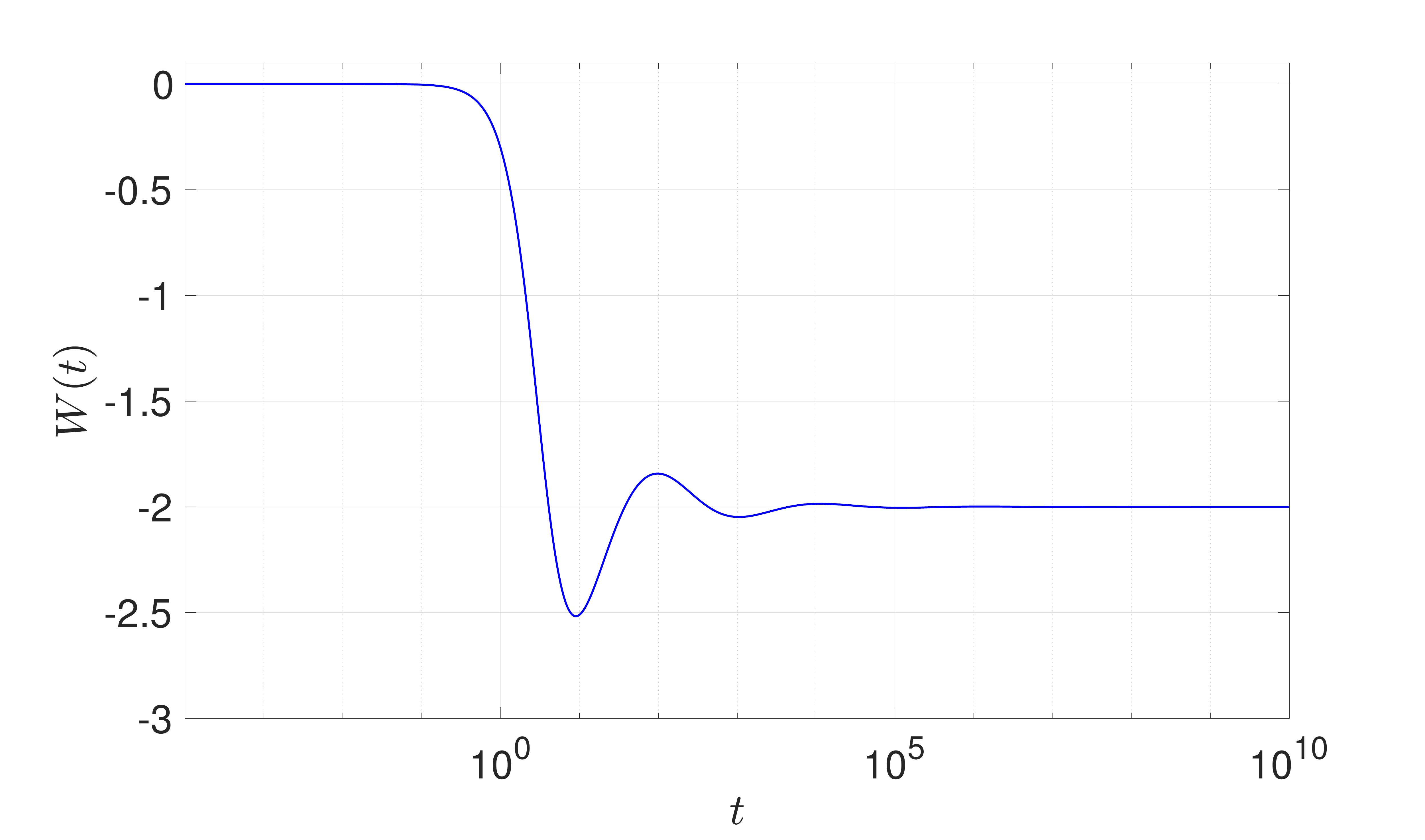}
}
\caption{Profiles of the functions $U(t)$ (left) and  $W(t)$ (right) on the interval $[10^{-1}, 10^5]$, using a base $10$ logarithmic scale for the $x$-axis.}
\label{fig:U-W}
\end{figure}
Moreover,  solutions to~\eqref{y_d} are in one-to-one correspondence with positive solutions  of the numerical equation
\begin{equation} \label{Nequ}
W(t) = -N ,
\end{equation}
where $W:= t \mapsto t \, U'(t)$.

A qualitative description of  $W$
and its asymptotic behavior can be obtained by means of classic dynamical systems techniques. This was done, for instance, in~\cite{chandr,emden,gel,jo-lun}.
Indeed, with the change of variables
\begin{equation*}
v(s)=u(e^s)+2\, s \, , \quad s \in \RR ,
\end{equation*}
problem~\eqref{IVPprob} turns into the second-order autonomous problem
\begin{equation}\label{auto}
\left\{ \begin{array} [c]{l}
 v'' + v' + e^v - 2 = 0, \\[2mm]
\displaystyle \lim_{s\to -\infty}v(s) =-\infty, \\[2mm]
\displaystyle  \lim_{s\to -\infty}v'(s) = 2.
\end{array}
\right.
\end{equation}
Phase-plane analysis shows that the derivative of the solution of~\eqref{auto}
oscillates about and tends to $0$ as $s\to\infty$.
Returning to the original variables, this means that $W$
oscillates about and tends to $-2$ as $t \to \infty$.

As a modern alternative to the argument above, standard techniques of numerical integration, directly applied to the initial value problem~\eqref{IVPprob}, yield the profile of $W$ (see Figure \ref{fig:U-W}, right).
Both approaches lead to the following result.
\begin{theo}\label{main}
There exist $N_1, N_2 \in (0,\infty)$, with $N_1 < 2 < N_2$, such that
\begin{description}
\item[(a)] problem~\eqref{y_d} has solutions if and only if $N \le N_2$;
\item[(b)] problem~\eqref{y_d} has a unique solution if and only if $0< N < N_1$;
\item[(c)] the number of solutions to problem~\eqref{y_d} tends to infinity as $N \to 2$.
\end{description}
\end{theo}
Accurate values of $N_1$ and $N_2$ have been obtained by exploiting the pathfollowing algorithm presented in Section \ref{sec:5}, as well as by solving problem \eqref{IVPprob} in Matlab with the aid of some available ODE solvers coupled with the event location tool to detect when the derivative of $W(t)$ vanishes during the computation of the numerical solution.  All these numerical computations show that $N_1 \approx 1.8427$ and $N_2 \approx 2.5175$.

\section{Discretization}
\label{sec:4}

In the sequel, we will assume $\Pi_1=\Pi_4=1$ in \eqref{y_d}, thus obtaining the simplified problem
\begin{subequations}\label{Nprob}
\begin{alignat}{1}
& y''+\frac{2}{\eta} \, y' +e^{y}=0, \label{equat} \\[1mm]
& y'(0)=0, \label{symm} \\[1mm]
& y'(1)=-N. \label{neum}
\end{alignat}
\end{subequations}
This choice of the parameters entails no loss of generality: if $y$ solves~\eqref{Nprob}, then the function
\begin{equation*}
\widetilde y(\eta) = y(\eta/\Pi_4) - \log (\Pi_1 \, \Pi_4^2) , \quad \eta \in [0,\Pi_4] ,
\end{equation*}
solves \eqref{y_d}. (See the Appendix for more general invariance properties of \eqref{y_d}.) From a computational point of view, the numerical schemes we  introduce to compute approximated solutions work well for this choice of the parameters.
\begin{rem}
In view of Theorem \ref{rep}, a numerical approximation to the solution of (\ref{Nprob}) could be in principle obtained by solving the algebraic equation (\ref{Nequ}). Although successful for theoretical purposes, this strategy brings out a number of difficulties when approaching the problem numerically. When $N$ is close to the limit value $2$, equation (\ref{IVPprob}) has to be solved on a sufficiently large integration interval that encloses all the  solutions of (\ref{Nequ}). However, the stiff nature of  problem (\ref{IVPprob}) for large value of $t$  poses  accuracy issues.   It turns out that equation (\ref{Nequ}) is actually affected by a perturbation whose size increases with $t$. On the other hand, since $W'(t) \rightarrow 0$ as $t$ increases (see Figure \ref{fig:U-W}, right),  the conditioning number of the problem also increases with $t$, thus preventing an accurate approximation of the solutions of  (\ref{Nequ}).
\end{rem}
We have solved problem \eqref{Nprob} by means of two  different numerical approaches described in the subsections below. A comparison of the results has allowed us to prevent possible discrepancies in the numerical approximations that might emerge after the discretization process. 

Both techniques need an initial guess of the solution to be used as starting point in the iterative procedure associated with the numerical method employed to solve (\ref{Nprob}). A consistent initial profile mimicking the shape of the solution has been obtained by means of a polynomial of degree two 
\begin{equation}
\label{initcond}
P(\eta)=a\eta^2+b\eta+c.
\end{equation} 
Imposing the boundary conditions \eqref{symm}-\eqref{neum} yields
\begin{equation}
\label{initcond1}
0=P'(0)=b, \qquad -N=P'(1)=2a+b \Rightarrow a=-\frac{N}{2},
\end{equation} 
thus $P(\eta)=-\frac{N}{2}\eta^2+c$. The remaining free coefficient $c$ has been computed  following two different strategies, as described below. 

To begin with, note that the characteristic numbers defined in \eqref{cn} satisfy the equality
\begin{equation}\label{total_mass1}
\int_0^{\Pi_4} \xi(\eta) \eta^2 \mathrm{d} \eta = N \dfrac{\Pi_4}{\Pi_1},
\end{equation} 
which follows from the expression of the gas mass
\begin{equation*}
    m_g = 4 \pi \int_{0}^a  \rho(u)u^2\; \mathrm{d} u
\end{equation*}
after introducing the change of variables in \eqref{scal}. With $\Pi_1=\Pi_4=1$, \eqref{total_mass1} reduces to
\begin{equation}
\label{total_mass2}
\int_0^{1} \xi(\eta) \eta^2 \mathrm{d} \eta = N.
\end{equation}

\begin{itemize}
\item[(a)] 
By imposing that $e^{P(\eta)}$ satisfies \eqref{total_mass2}, we get
$$
N=\int_0^1 \eta^2e^{P(\eta)}\mathrm{d}\eta~=~e^c\int_0^1 \eta^2e^{-\frac{N}{2}\eta^2}\mathrm{d}\eta,
$$
whence 
\begin{equation}
\label{initcond2a}
c=\log\left(\frac{N}{\int_0^1 \eta^2e^{-\frac{N}{2}\eta^2}\mathrm{d}\eta}\right).
\end{equation} 
The integral in \eqref{initcond2a} can be easily computed by a quadrature formula.
\item[(b)]Since \eqref{total_mass2} implies 
$\int_0^1  \eta^2 \left(\xi(\eta) - 3N \right) \mathrm{d} \eta = 0$, there exists $\eta_0 \in [0,~1] $ such that
$\xi(\eta_0) = 3N$.
For example, assuming $\eta_0=\frac{1}{2}$, we get  
$3N=e^{P(\frac{1}{2})}=e^{-\frac{N}{8}}e^c$, 
whence
\begin{equation}
\label{initcond2b}
c~=~\frac{N}{8}+\log(3N).
\end{equation} 
\end{itemize}

\subsection{HOFiD code}
\label{subsec:4.1}
The MATLAB code HOFiD (see~\cite{ASjnaiam,ASaip}) is based on high order (up to 10) finite difference schemes to solve general second order ODEs,
\begin{equation} \label{eq:twpbvp}
f(\eta,y,y',y'') = 0, \quad \eta \in [a,b],
\end{equation}
subject to Neumann, Dirichlet or mixed boundary conditions. The ideas underlying this technique have also been  adapted to solve Sturm--Liouville and multi-parameter problems \cite{ALSWjamc,ALSWcpc}. 

To explain the procedure, let us, for simplicity, consider the interval $[a,b]$ uniformly subdivided into $n$ subintervals by means of $n+1$ equispaced grid points $\eta_i = a+ih$, $i = 0, \dots, n$, with the stepsize $h = (b-a)/n$. The derivatives at $\eta_i$ are approximated  as
\begin{equation} \label{eq:derapp}
y^{(\nu)}(\eta_i) \approx y_i^{(\nu)} = \frac{1}{h^{\nu}}
\sum_{j=-s_i}^{r_i} \alpha_{j+s_i}^{(\nu)} y_{i+j}, \quad \nu = 1, 2, \quad i=1,\dots,n{-}1,
\end{equation}
where the coefficients $\alpha_j^{(\nu)}$ involved in each formula are computed to guarantee the highest possible order of consistency.

When \eqref{eq:twpbvp} is coupled with  Dirichlet boundary conditions, we obtain a closed system with $n{-}1$ nonlinear algebraic equations
\begin{equation} \label{bvp:cond}
F(x_i,y_i,y_i',y_i'') = 0, \quad i = 1,\dots, n{-}1,
\end{equation}
for the $n{-}1$ unknowns $y_i$, $i = 1, \dots, n{-}1$.
In case of a Neumann boundary condition in $a$ (we proceed similarly in $b$), the value $y_0$ is considered as an unknown and the equation
\begin{equation*} 
\frac{1}{h}\sum_{j=0}^{r_0} \alpha_{j}^{\prime} y_i = y_0^{\prime} = y^{\prime}(a)
\end{equation*}
is added to \eqref{bvp:cond}.

The values $r_i$ and $s_i$ in (\ref{eq:derapp}) may be chosen according to the value of the index $i$. For the second derivative the choice $r_i=s_i$ leads to symmetric formulae of order $2r_i$. For the first derivative we choose $r_i=s_i\pm 1$, depending on the sign of the term multiplying $y_i'$, in accordance with the upwind strategy wich improves the stability domain.

These choices are not always possible but depend on the available points to the left and to the right of $i$.
For this reason, at the initial grid points $\eta_i$,  $i< \frac{p}{2}$, we choose $s_i=i$ and $r_i=p-s_i+\nu-1$, since $\frac{p}{2}$ values are not available on the left of $\eta_i$. Analogously, for $i>n-\frac{p}{2}$, we choose $r_i=n-i$ and $s_i=p-r_i+\nu-1$. Hence, in a block form each derivative is approximated by $Y^{(\nu)} = A_{\nu} Y$, where $A_{\nu}$ is a banded matrix \cite{ASjcam}.

For regular second order BVPs for ODEs with appropriately smooth solution, the global order of the above finite difference scheme coincides with the order of consistency \cite{ASjnaiam}.

The code employs fixed even orders from 2 to 10 (order 2 corresponds to the classical centered differences),  and variable stepsize based on an equidistribution of the error obtained by exploiting the solution yielded by two consecutive order formulae.

\subsection{Codes for first order BVPs}

A more standard approach consists in recasting the second order differential equation \eqref{equat}  as a first order system and then apply a general-purpose code for boundary-value problems. Setting $(z_1,z_2)^T = (y, y')^T$, equation \eqref{equat} assumes the equivalent  form
\begin{equation}
\label{first_order_problem}
\begin{array}{l}
z'_1(\eta) = z_2(\eta),\\[.3cm]
\displaystyle z'_2(\eta) =  - \frac{2}{\eta} z_2(\eta) +  e^{z_1(\eta)} ,
\end{array}
\end{equation} 
with $\eta \in[0,1]$, and coupled with the  boundary conditions
\begin{equation}
\label{f_o_bc}
z_2(0) = 0, \qquad z_2(1) = - {N},
\end{equation} 
To avoid the singularity, at $\eta= 0$ we use the equations
\begin{equation*}
z'_1(\eta) = z_2(\eta), \qquad z'_2(\eta) =   - \frac{e^{z_1(\eta)}}{3},
\end{equation*} 
which have been derived by considering the  first derivative of the equation
$$
\eta y''(\eta) + 2 y'(\eta) - \eta e^{y(\eta)} = 0 
$$
computed at $\eta = 0$.

The  first order BVP \eqref{first_order_problem}-\eqref{f_o_bc} has been solved in MATLAB by using the  codes {\tt bvptwp} \cite{ACMBVPTWP}, {\tt tom }\cite{TOM}, and the codes {\tt bvp4c} and {\tt bvp5c} of the MATLAB release R2014b. The same problem has also been solved  in R by using  the solver  {\tt bvptwp}  belonging to the R-package {\tt bvpSolve} \cite{bvpSolve}.
The codes {\tt bvptwp} and   {\tt tom } make use of a similar hybrid mesh selection strategy based on conditioning, which allows to get output information about the conditioning of the continuous problem \cite{CONDMESH1,CONDMESH2}.  

In order to obtain a solution in the lowermost branch of the bifurcation diagram in Figure \ref{fig:diagrams} (left plot), we have considered an  initial  mesh of 101  equispaced meshpoints, with a constant stepsize equal to  0.01, and an initial guess as in \eqref{initcond}-\eqref{initcond1}-\eqref{initcond2b}.

All the methods are able to find a solution  for  $N \le N_2 \simeq 2.5175513$. However, the input absolute and relative tolerances have been set equal to $10^{-8}$ since, with higher tolerances,  the two codes {\tt bvp4c} and  {\tt bvp5c}  would give a solution in correspondence to values of $N$ for which the continuous problem   admits no solutions. 

The condition number of the problem is reported in Table \ref{tab_cond} for different values of $N$,  showing that the problem is in general well  conditioned. As is expected, the condition number grows when $N$  approaches the limit value  $N_2$.  \begin{table}
\centerline{ 
\begin{tabular}{c|lllllll}
N & 0.5  & 1    & 1.5   & 2     & 2.5   & 2.51  &  2.5175513  \\		
\hline
cond. & 3.87 & 2.74 &  2.64 & 4.04  & 44.02 & 97.04 &  $4.61 \cdot 10^{4}$  
\end{tabular} }
\caption{Condition numbers for different values of $N$}
\label{tab_cond}
\end{table}


\section{Numerical results}
\label{sec:5}
In view of Theorems \ref{rep} and {\ref{main}}, which show how the parameter $N$ affects the existence and  multiplicity of the static configurations of the physical model, the most relevant question to be addressed from a numerical viewpoint may be stated as follows:
\begin{quote} 
{\em
find out (possibly all)  the isolated solutions of (\ref{Nprob}) corresponding to  a given value of the parameter $N \le N_2$ (existence condition).} 
\end{quote}

Assume that the nonlinear system obtained after one of the discretization procedures illustrated in Section \ref{sec:4} has dimension $n+1$ and generates an approximation $y=(y_0,y_1,\dots,y_n)^\top$ to the true solution, where $y_0\approx y(0)$ and $y_n\approx y(1)$. Then, the dependence of the solution $y$  on the parameter $N$ in  (\ref{Nprob}) may be visualized, in the parameter space $\RR^{n+2}$, as a one-dimensional curve  forming a branch of solutions. The occurrence of multiple solutions prevents us from using a simple continuation argument on the parameter $N$ to determine the complete bifurcation diagram. 

To highlight this difficulty one can, for example, resort to the projection of the curve onto the plane $(y_0,N)$ or the plane $(y_n,N)$. The resulting diagrams are reported in Figure \ref{fig:diagrams} and reveal the presence of turning points where ${dy_0}/{dN}\rightarrow \infty$ (left plot) or, equivalently, ${dy_n}/{dN}\rightarrow \infty$ (right plot). Table \ref{tab_cond} accounts for the difficulty of solving (\ref{Nprob}) in the vicinity of $N_2$. Similar ill-conditioning issues arise when $N$ is close to any subsequent singular points, which makes the numerical treatment of (\ref{Nprob}) challenging.
\begin{figure}
\centerline{
\includegraphics[width=.56\textwidth]{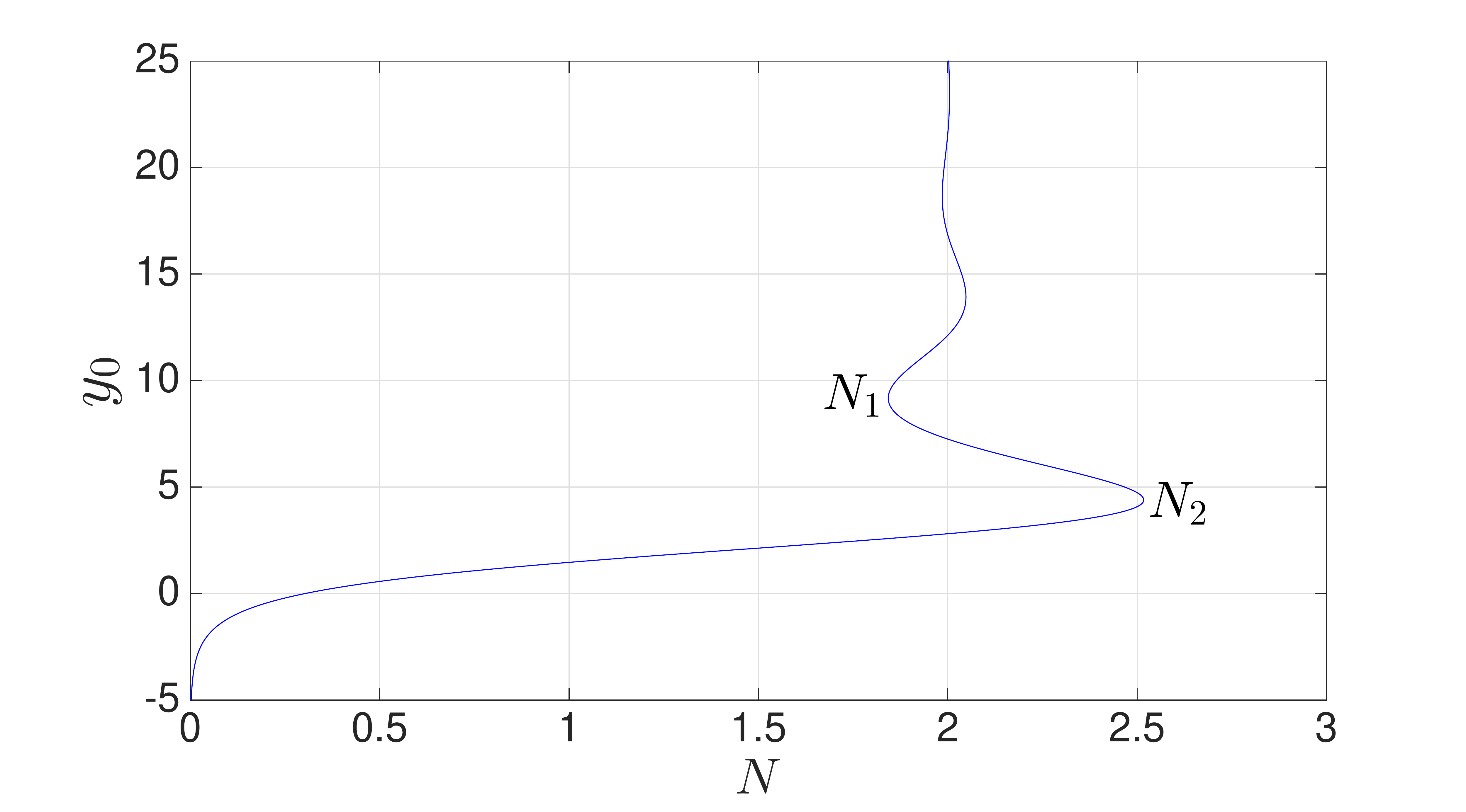} \hspace*{-.8cm}
\includegraphics[width=.56\textwidth]{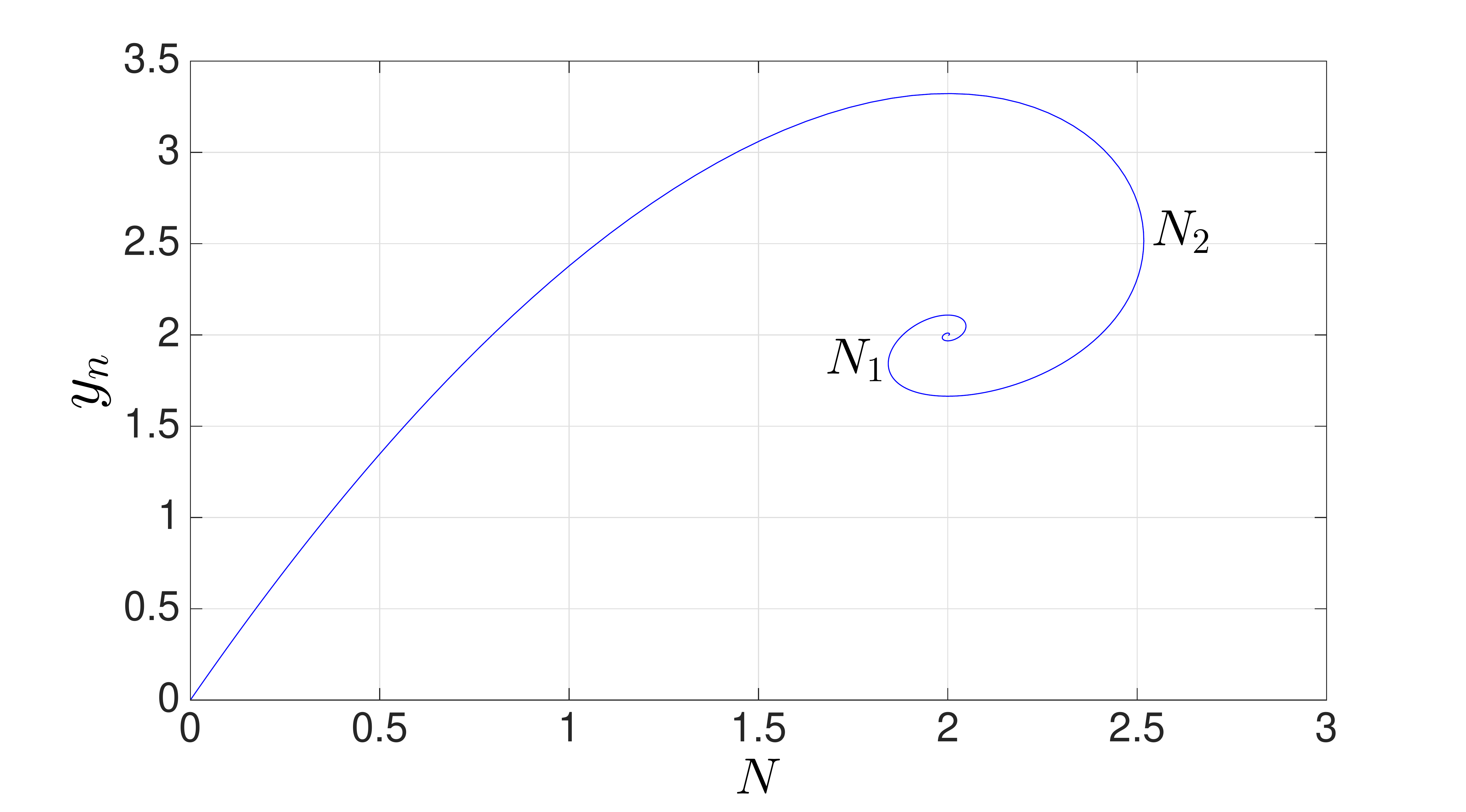}
}
\caption{Bifurcation diagrams showing the dependence of the solution of problem (\ref{Nprob}) on the parameter $N$. Specific continuation techniques are required to jump over the turning points.  \label{fig:diagrams}}
\end{figure}

The  branch starting at $N=0$ and ending at $N=N_2$ contains the configurations which have been recognized as thermally stable in \cite{ESA19a}. We have experienced that the corresponding solutions are those obtained by the codes we have employed starting with generic initial guesses as those illustrated at the beginning of Section \ref{sec:4}. According to \cite{ESA19a}, the other branches, for example the one starting at $N_2$ and ending at $N_1$, contain thermally unstable configurations. To get the corresponding numerical approximations we need very accurate initial guesses.  Consequently, a suitable numerical continuation technique must be implemented to jump over the turning points and search for the solutions which lie on other, thermally unstable branches. 

The left picture in Figure \ref{fig:diagrams} suggests that a natural way to avoid turning points could be using $y_0$, instead of $N$, as a parameter. This would change problem (\ref{Nprob}) into an initial value problem with initial conditions $y_0:=y(0)\in [0,+\infty)$ and $y'_0:=y'(0)=0$. However, we did not follow this approach since, as was outlined above in the definition of the task, we are interested in computing solutions corresponding to a prescribed value of $N$. Therefore, we should couple the initial value problem with a suitable shooting technique; this would be less efficient than employing a continuation strategy involving the parameter $N$ directly. In addition, we observe that for large values of $y_0$ the problem becomes stiff and requires a large number of mesh points. 

The numerical continuation strategy we have adopted may be thought of as a compromise between solving the original problem (\ref{Nprob}) and the initial value problem described above. It  consists of a parametrization involving both $y_0$ and $N$ and defined as follows:

\begin{itemize}
\item[(1)] Set $N^{(1)}:=N$ and compute the first solution $y^{(1)}$ of (\ref{Nprob}) corresponding to the point $z_1:=(N,y_0^{(1)})$ in the lower branch of the left picture in Figure \ref{fig:diagrams}. If $N<=N_2$, this solution always exists.
\item[(2)] Set $N^{(2)}:=N+\Delta N$, with $\Delta N$ small enough, and compute the solution $y^{(2)}$ of (\ref{Nprob}) with boundary condition $y'(1)=-N^{(2)}$. Set $z_2:=(N^{(2)},y_0^{(2)})$. 
\item[(3)] \underline{Predictor}: for a given, small enough stepsize $\delta>0$, compute an initial guess $w$ for the new point on the branch as 
\begin{equation*}
w:=z_2+\delta v, \qquad \mbox{with }~ v:=\frac{z_2-z_1}{||z_2-z_1||_2}. 
\end{equation*}
\item[(4)] \underline{Corrector}: Solve (\ref{equat})-(\ref{symm}) and  replace (\ref{neum}) with the following two extra conditions 
\begin{equation}
\label{new_bc}
y(0)=w(2)+cv(1),\qquad -y'(1)=w(1)-cv(2),
\end{equation}
which, removing the extra unknown $c\in \RR$, simplifies as
\begin{equation}
\label{new_bc1}
v(2)y(0)-v(1)y'(1)= v(1)w(1)+v(2)w(2).
\end{equation}

\item[(5)]  If $y=(y_0,\dots,y_n)^\top$  is the resulting numerical solution and $y'=(y'_0,\dots,y'_n)^\top$ contains the corresponding approximations of the first derivatives given in output from the previous step,  set 
$$
(N^{(1)}, y_0^{(1)}):=(N^{(2)}, y_0^{(2)}) \qquad \mbox{and} \qquad (N^{(2)}, y_0^{(2)}):=(-y'_n, y_0).
$$ 
\noindent Then go to step (3) to advance the approximation of the branch.
\end{itemize}
A few comments elucidating the geometric meaning of the above steps are in order. Step (1) computes the basic solution belonging to the thermally stable branch. As was mentioned above, if $N$ is sufficiently away from the limit value $N_2$, the codes  easily  solve problem (\ref{Nprob}) directly. Should $N$ be close to $N_2$, we could start with selecting a lower value for $N^{(1)}$ and then reach $N$ by continuation, as illustrated in the subsequent steps. 

Step (2) computes a solution close to but different from the one obtained at the previous step.  This further approximation is needed  to start the secant predictor-corrector method employed at steps (3) and (4).

At step (3) we predict the next point on the branch by linear extrapolation based on the last two computed solutions. Notice that $v$ is the unitary vector defining the direction of the line joining $z_1$ and $z_2$, in the $(N,y_0)$-plane.

Step (4) implicitly corrects the initial guess $w$ by solving a modified boundary value problem with (\ref{neum})  replaced by (\ref{new_bc1}), which comes from (\ref{new_bc}). Here, the couple $(w(1)-cv(2),w(2)+cv(1))$, with $c\in \RR$,  locates a generic point on the line orthogonal to $v$ and passing through the point $w$, in  the $(N,y_0)$-plane. As initial guess for the solver we use the configuration obtained at the previous step which, for $\delta$ small enough, assures a fast convergence of the nonlinear iteration scheme.  

The numerical continuation technique summarized at steps (3)-(4) is iterated and, if needed, the value of  the stepsize $\delta$ at step (3) is tuned in order to assure that the numerical continuation procedure may successfully compute the current branch, go beyond the corresponding turning point and continue with approximating the  subsequent branch in the bifurcation diagram. During the computation of this new branch, two different situations may occur:
\begin{itemize}
\item[-] At a certain step on the new branch, the following event is recognized: $N\in [\min\{N^{(1)}, N^{(2)}\},\max\{N^{(1)}, N^{(2)}\} ]$. If the last used stepsize $\delta$ is small enough, the code will be able to solve the original problem (\ref{Nprob}) directly using one of the two last computed solutions as initial guess. After this new solution has been obtained, the continuation algorithm is started again to reach a new branch. 
\item[-] A new turning point is achieved without meeting the event described above. In view of the shape of the bifurcation diagram (see the left picture in Figure \ref{fig:diagrams}), the algorithm may be stopped since no new solutions may be attained in the subsequent branches. 
\end{itemize}

Concerning the choice of the parameter $\delta$ at step (3) of the pathfollowing algorithm, an automatic stepsize variation strategy is mandatory to reproduce a correct shape of the bifurcation diagram, thus  preventing that turning point might be skipped. In more detail, denoting by $\delta_k$ the value assumed by the parameter $\delta$ at step $k$ and by $n_k$ the number of iterations needed by the solver of the nonlinear algebraic equation associated with the discrete problem, the following scheme has been implemented:
\begin{itemize}
\item[(a)] Set $\delta_n=\delta_{n-1}$.
\item[(b)] Solve the discrete problem.
\item[(c)] If $n_k>n_{\max}$ or convergence is not achieved, then set $\delta_n=\max\{\delta_n/2, \delta_{\min}\}$. In particular, in case of convergence failure, go back to step (b). 
\item[(c)] If $n_k<n_{\min}$  then set $\delta_n=\min\{2\delta_n,\delta_{\max}\}$.
\end{itemize}
Here, we constrain $\delta_n$ to lie inside a suitable interval $[\delta_{\min},\delta_{\max}]$ to smoothly capture the correct curvature of the branches in a neighbourhood of turning points.\footnote{Actually, the value $\delta_{\min}$, below which the procedure would be stopped returning a failure flag, has never been attained in the experiments included in the paper.} We emphasize that the above procedure, which turns out to be successful for the problem at hand, is mainly addressed at guaranteeing a (fast) local convergence of the iteration scheme towards the solution closest to the initial guess. In fact, we have experienced that, as we move along the upper branches of the bifurcation diagrams, the parameter $\delta_n$ is gradually reduced to ensure that the initial profile of the numerical solution be inside the basin of attraction of the Newton iteration scheme (for a  more sophisticated numerical pathfollowing technique see, e.g., \cite{deficu}).

Hereafter, we illustrate the output of the continuation algorithm coupled with one of the two solvers for BVPs described in Section \ref{sec:4}. Since these two codes yield essentially the same results, we only report the results for the former, namely the HOFiD code described  in Section \ref{subsec:4.1}.\footnote{The use of the latter code and a comparison between the corresponding numerical approximations turned out to be very useful in assessing the correctness of the obtained results.} For such a code, in the following test problems,  we used the default order 6 for the methods, error tolerance $10^{-8}$ as the required accuracy in the numerical approximation, and 21 initial equispaced points to solve the first problem. Moreover, we fixed $\delta_0=0.2$ as starting value of the parameter $\delta$ needed by the continuation strategy.
Besides the results illustrated below, the continuation technique allowed us to compute $N_2\simeq 2.51755148$ as the largest value of $N$ (with 8 significant digits) for which a solution exists (see Figure \ref{fig:diagrams} and Theorem \ref{main}).

\subsection{Test \#1}  
As a first, simpler example, we set $N=1.9$ and let the code compute all the corresponding solutions to (\ref{Nprob}). To start the continuation procedure, we compute the solutions belonging to  the lower branch of the bifurcation diagram associated with the values $N=1.8$ and $N=1.9$, respectively. 

\begin{figure}[htb]
\centerline{
\includegraphics[width=.56\textwidth]{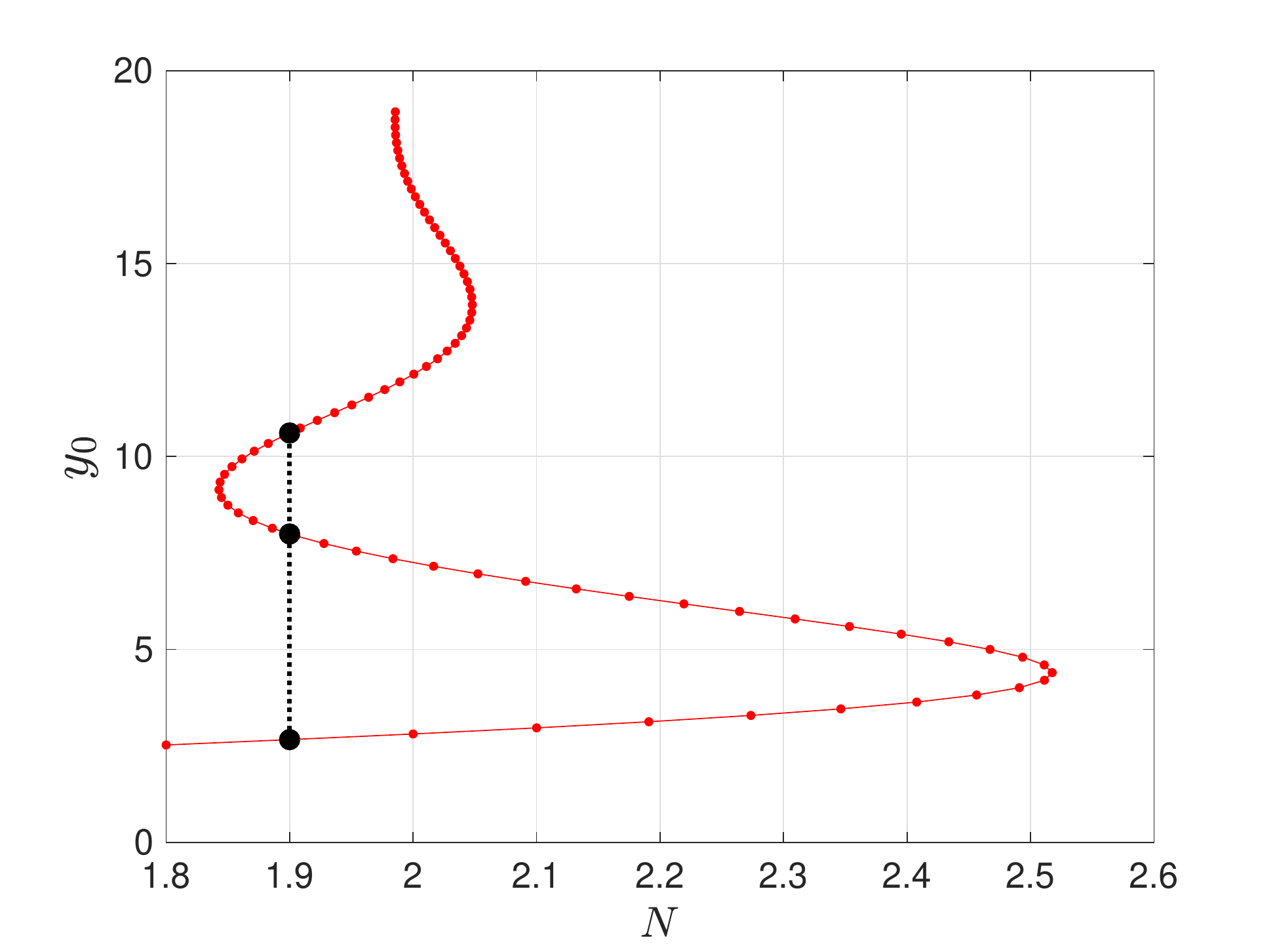} \hspace*{-.8cm}
\includegraphics[width=.56\textwidth]{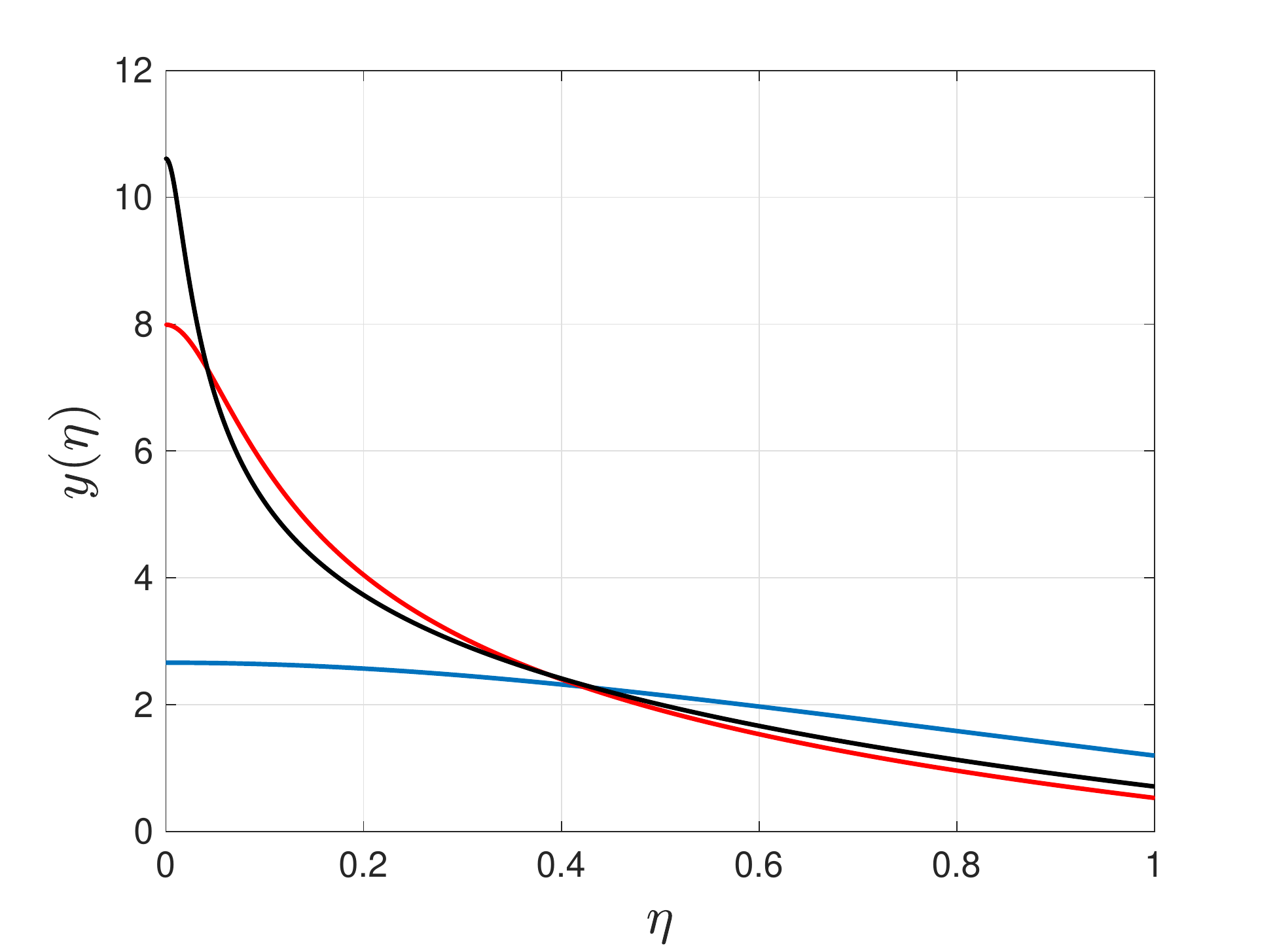} }
\caption{Computation of the three solutions to (\ref{Nprob}) for $N=1.9$. Bifurcation diagrams of the computed solutions with the continuation technique (dots in the left picture) and the three numerical solutions (right picture).\label{fig_ex19}}
\end{figure}
\begin{table}[bh]
\centerline{ 
\begin{tabular}{l|ccc}
$y_0$ & 2.6618 &  7.9906 &  10.609   \\
$n+1$   &    51 &  1248 &  2600 \\
$h_{\min}$ & 1.71e-2 &  4.72e-4 &  5.90e-5
\end{tabular} }
\caption{The 3 solutions obtained for $N=1.9$.}
\label{tab_ex19}
\end{table}
Figure  \ref{fig_ex19} shows a set of 84 dots on the bifurcation diagram computed by the algorithm. This means that, with the selected parameters, the code requires the computation of the solution of 84 BVPs before stopping. It should be noticed that a few iterations of the nonlinear solver are needed at each step, since we are using the previous computed solution as initial guess.  

The three isolated solutions corresponding to $N=1.9$ are displayed in the right picture of Figure  \ref{fig_ex19}. For each solution, Table \ref{tab_ex19} reports the value $y_0\approx y(0)$ of the solution at $\eta=0$, the number of points $n+1$ introduced by the BVP solver, and the minimum stepsize $h_{\min}$ used by the variable stepsize technique employed in the code. We notice that, after the computation of the third (uppermost) solution, the code needs to proceed with the detection of the two subsequent turning points before realizing that no further solution does exist. 

\subsection{Test \# 2}
In this second example, we solve a more involved problem, by choosing for the parameter $N$ a value very close to $2$, namely, $N=2.0001$. We start with computing the solutions belonging to the lower branch of the bifurcation diagram associated with the values $N = 2$ and $N = 2.1$, respectively. 
In such a case, the continuation algorithm finds eight isolated solutions before stopping.  Figure \ref{fig_ex20001} shows the bifurcation diagram with 202 dots representing the as many boundary value problems needed by the procedure. Table \ref{tab_ex20001} has the same meaning as Table \ref{tab_ex19};  we notice that, as we move from one solution to the next in the subsequent upper branch, the number of points needed by the code to approximate the solution within the input tolerance significantly increases. This is a consequence of the formation of a boundary layer in a right neighbourhood of 0, whose radius on the $\eta$-axis decreases as we move from lower to upper branches.  
The use of an efficient variable stepsize strategy, with the possibility of increasing the number of points to match the required accuracy, becomes crucial to handle the increasing stiffness of the problems to be solved during the branches inspection.  
\begin{figure}[hbt]
\centerline{
\includegraphics[width=.56\textwidth]{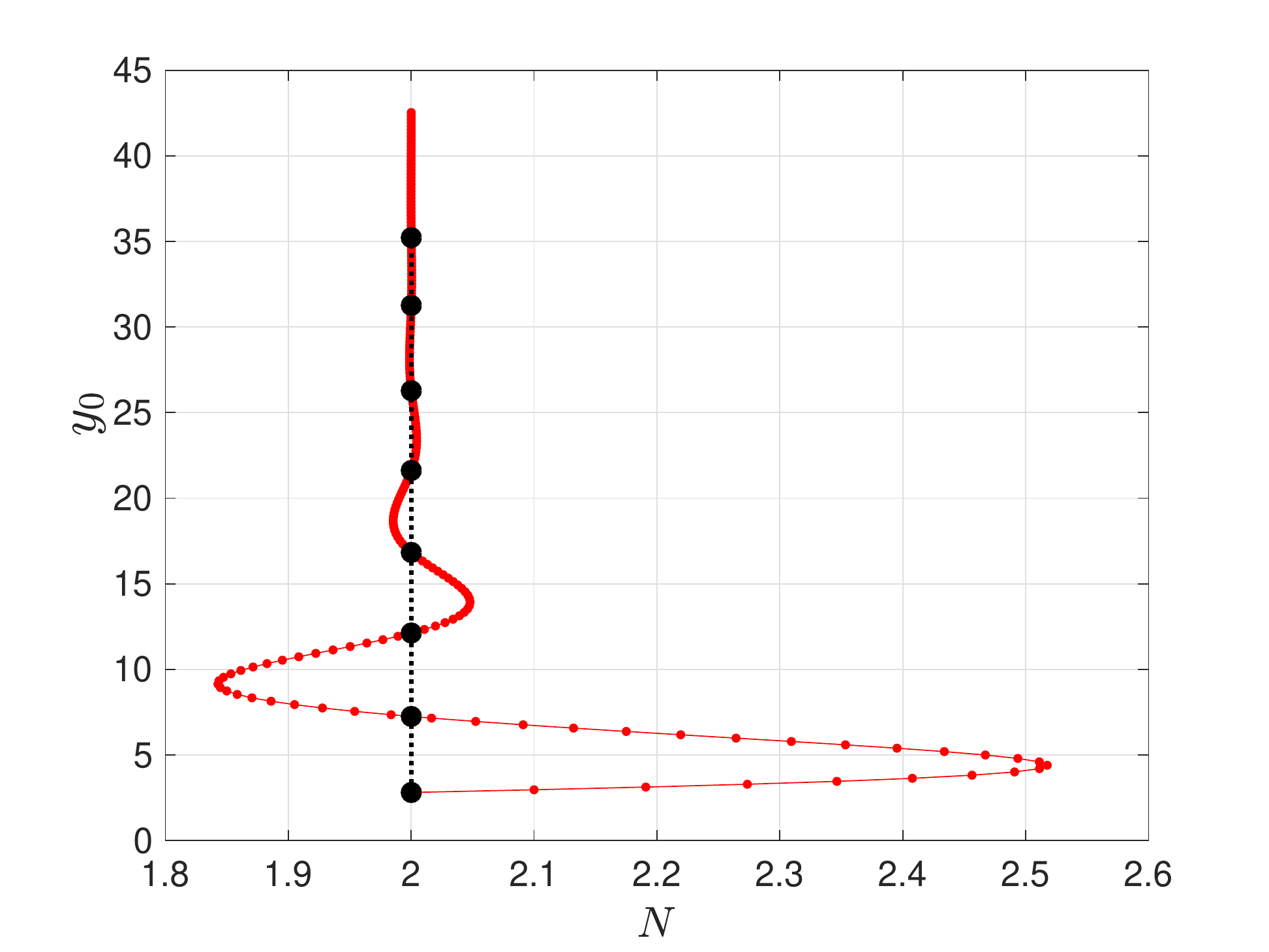} \hspace*{-.8cm}
\includegraphics[width=.257\textwidth]{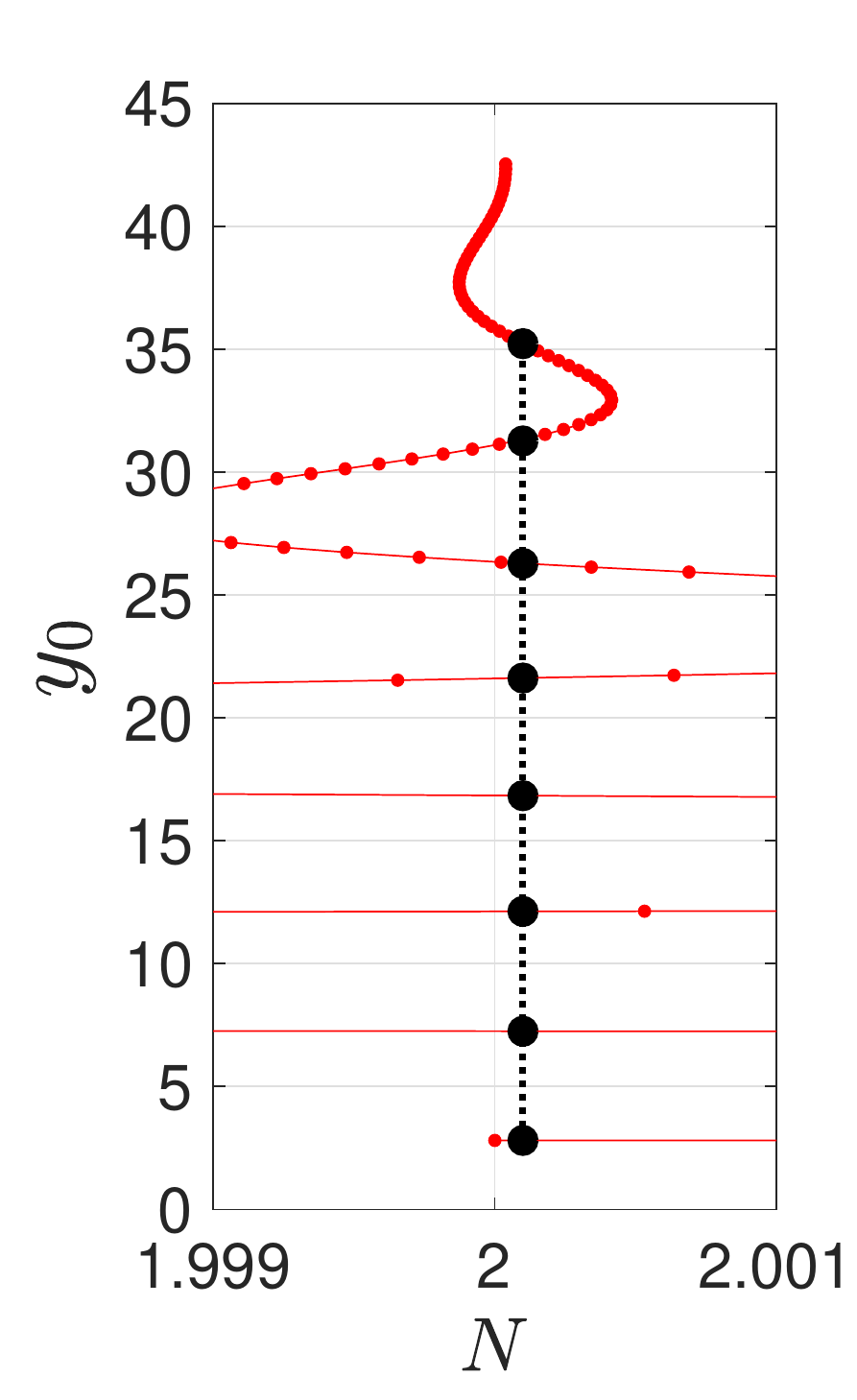} }
\caption{Bifurcation diagram with a closeup highlighting the  eight computed solutions to (\ref{Nprob}) for $N=2.0001$. \label{fig_ex20001}}
\end{figure}
\begin{table}[hbt]
\centerline{ 
\begin{tabular}{l|rrrrrrrr}
$y_0$ & 2.8082 &  7.2495 &  12.124 &  16.832 &  21.618 &  26.280 &  31.263 &  35.221 \\
$n+1$   & 51 &  306 &  1497 &  2163 &  3431 &  5048 &  7641 &  6399 \\
$h_{\min}$ & 1.67e-2 &  1.58e-3 &  4.17e-5 &  1.74e-6 &  4.14e-7 &  1.64e-8 &  6.60e-10 &  3.96-10
\end{tabular} }
\caption{Some statistics on the eight solutions to (\ref{Nprob}) obtained for $N=2.0001$.}
\label{tab_ex20001}
\end{table}

\begin{figure}[htb]
\centerline{
\includegraphics[width=.56\textwidth]{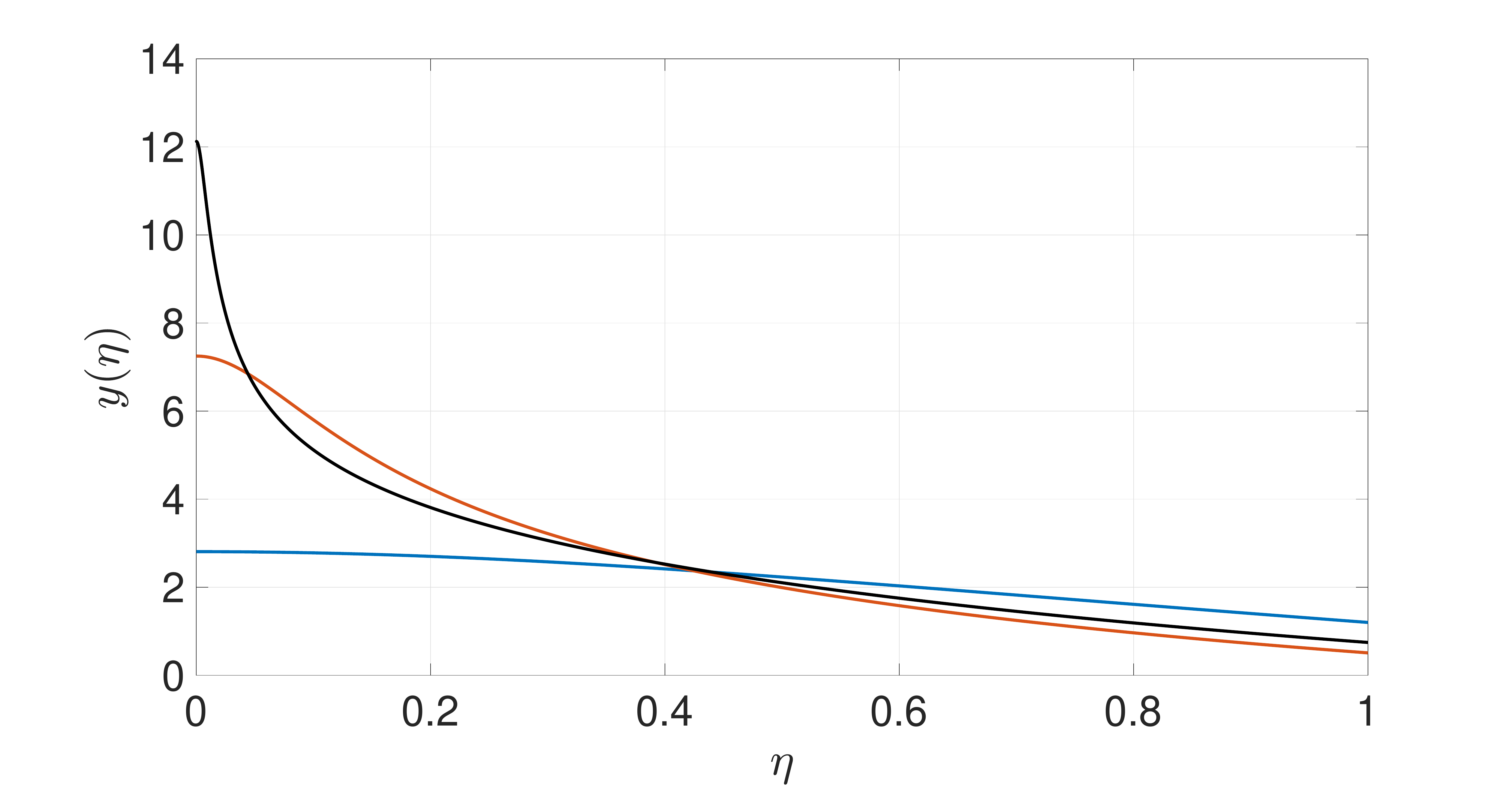} \hspace*{-.8cm}
\includegraphics[width=.56\textwidth]{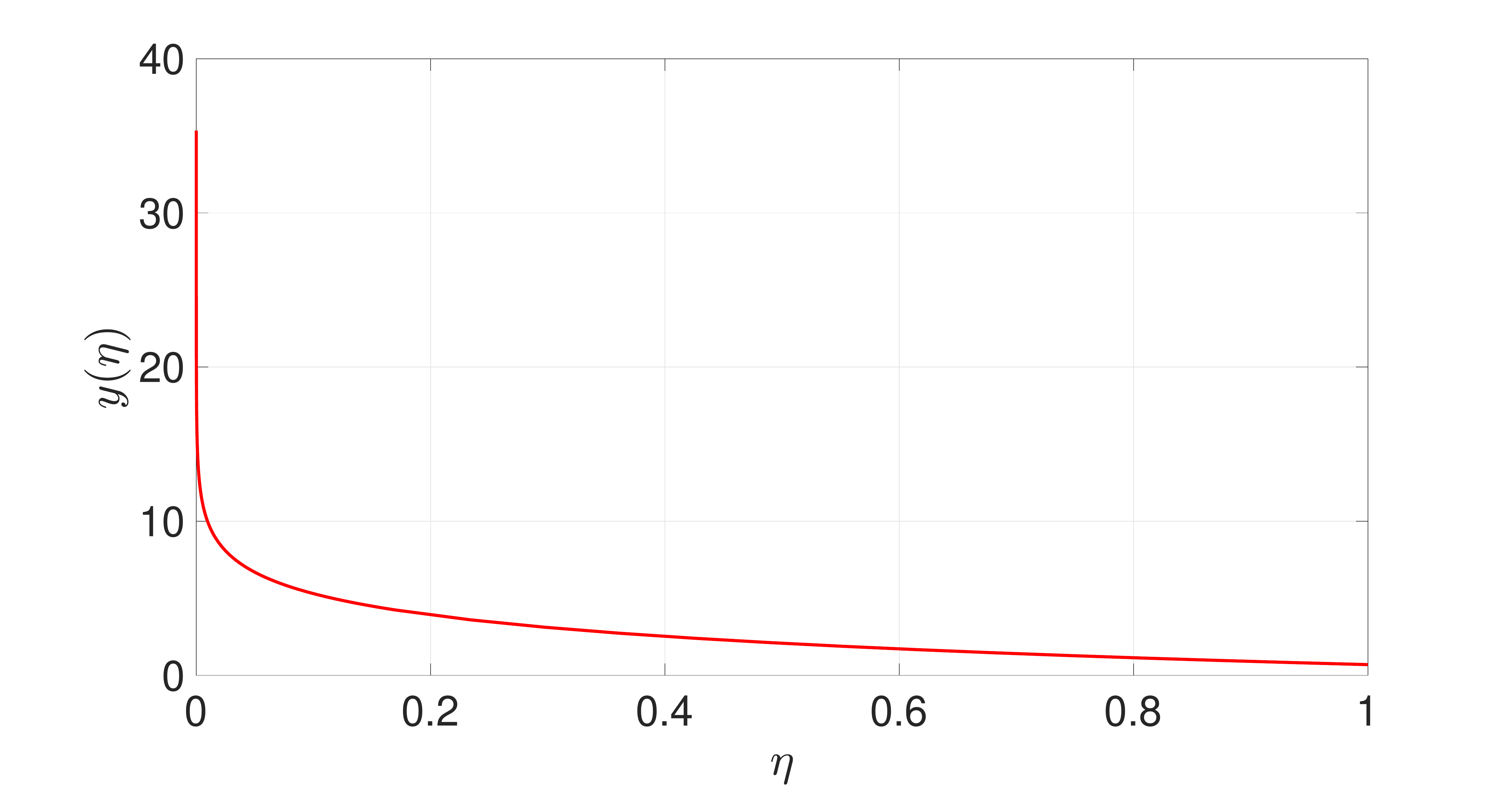} }
\caption{Computation of the first three solutions (left picture) and the last one (right picture) to (\ref{Nprob}) for $N=2.0001$. \label{fig_sol20001}}
\end{figure}

\section{Conclusions}
\label{sec:6}
In this paper, we have considered some mathematical aspects of a physical model governing the fluid-static configurations of a self-gravitating perfect gas enclosed in a spherical solid shell. The problem takes the form of the well-known  Lane-Emden equation describing the mass-density of the gas. Differently from what is usually done in  astrophysics, here the equation is coupled with boundary conditions emerging from the application of  Gauss' theorem at the center and at the internal wall of the solid shell. We have carried out an analytical study about the existence and multiplicity of solutions as well as an implementation of some robust numerical strategies to compute them. An extension to higher dimension has  been proposed in the Appendix.

\appendix
\section{Appendix}

Equation~\eqref{y_d_equat} is the three-dimensional radial version of the elliptic equation
\begin{equation}\label{ell-eq}
\Delta u+\lambda \, e^{u}=0
\end{equation}
(with $\Pi_1$ replaced by $\lambda$). In this section we generalize the results in Section~\ref{sec:3} to arbitrary dimensions and compare them with analogous, known results for Dirichlet problems.

We begin by observing that, for every $\lambda\in(0,\infty)$, the elliptic equation~\eqref{ell-eq}
has exactly a one-parameter family of entire radially symmetric solutions, defined by
\begin{equation}\label{app-repr}
u_\sigma(x)=U_{n}(\sigma\left\vert x\right\vert )+\ln\left({\sigma^{2}}/{\lambda}\right) ,
\end{equation}
where $\sigma \in (0,\infty)$ and $U_{n}$ is the unique maximal forward solution of the initial value problem
\begin{equation*}
\left\{
\begin{array}{l}
u''+\dfrac{n-1}{t} \, u' + e^{u}=0,  \\[2mm]
u'(0)=u(0)=0.
\end{array}\right.
\end{equation*}
The same argument as in Section~\ref{sec:3} shows that $U_n$ is global and decreasing throughout. The dimension $n$ does not seem to affect the profile of $U_{n}$ (see Figure~\ref{fig-U-n}, obtained via standard techniques of numerical integration); nonetheless, it plays a crucial role in determining the number of solutions for boundary value problems associated with equation~\eqref{ell-eq}.
\begin{figure}[htb]
\centerline{
\includegraphics[width=.8\textwidth]{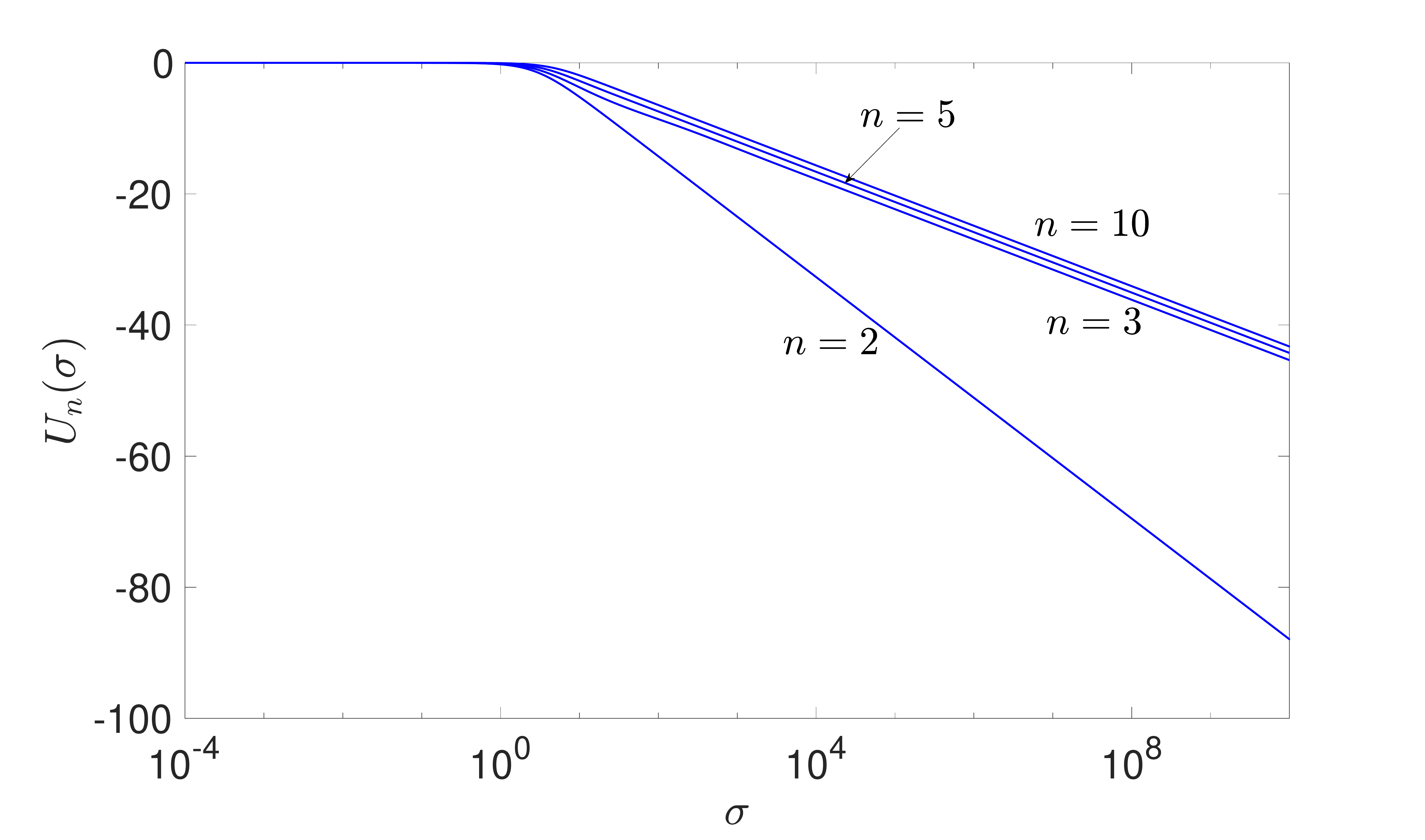} }
\caption{Profile of $U_{n}$ for different values of $n$.}
\label{fig-U-n}
\end{figure}

\subsection{Boundary value problems}
For $R\in(0,\infty)$, let $B_R=\left\{  x\in\RR^{n}\mid\left\vert x\right\vert <R\right\}$.
We can complement equation~\eqref{ell-eq}, posed on $B_R$, with the Neumann boundary condition
\begin{equation}\label{app-n}
\dfrac{\partial u}{\partial\nu} =\gamma\qquad\text{on }\partial B_R.
\end{equation}
In view of the comments above, we assume $\gamma\in(-\infty,0)$. The parameter $\lambda$ is irrelevant: if $u$ solves~\eqref{ell-eq}-\eqref{app-n} with $\lambda = \lambda_1$, then $u + \ln\left({\lambda_1}/{\lambda_2}\right)$ solves the same problem with $\lambda=\lambda_2$. Therefore, there is no loss of generality in confining attention to $\lambda=1$. Furthermore, if $u$ solves~\eqref{ell-eq}-\eqref{app-n} with $R = R_1$ and $\gamma=\gamma_1$, then the function
\begin{equation*}
x \mapsto u\left(\frac{R_1}{R_2} \, x\right) + \ln\left(\frac{R_1^2}{R_2^2}\right)
\end{equation*}
solves the same problem with $R = R_2$ and $\gamma=\gamma_1 R_1 / R_2$. Thus, we fix $R=1$ and discuss the existence and multiplicity of solutions to
\begin{equation}\label{app-n-2}
\left\{
\begin{array}{rl}
\Delta u+ e^{u}    =0 & \quad\text{in }B_1,\\[1mm]
\frac{\partial u}{\partial\nu}  =\gamma & \quad\text{on }\partial B_1
\end{array}\right.
\end{equation}
in dependence of $\gamma$ and $n$. We can also complement equation~\eqref{ell-eq} with the Dirichlet boundary condition
\begin{equation}\label{app-d}
u =c\qquad\text{on }\partial B_R ,
\end{equation}
with $c \in \RR$.
If $u$ solves~\eqref{ell-eq}-\eqref{app-d} with $\lambda=\lambda_1$, $R = R_1$, and $c=c_1$, then the function
\begin{equation*}
x \mapsto u\left(\frac{R_1}{R_2} \, x \right) + \ln\left(\frac{\lambda_1}{\lambda_2} \frac{R_1^2}{R_2^2}\right)
\end{equation*}
solves the same problem with $\lambda=\lambda_2$, $R = R_2$, and $c=c_1+ \ln\left(\frac{\lambda_1}{\lambda_2} \frac{R_1^2}{R_2^2}\right)$.
In line with the existing literature, we fix $c=0$ and $R=1$, and
discuss the existence and multiplicity of solutions to
\begin{equation}\label{app-d-2}
\left\{
\begin{array}{rl}
\Delta u+ \lambda \, e^{u}    =0 & \quad\text{in }B_1,\\[1mm]
u  =0 & \quad\text{on }\partial B_1
\end{array}\right.
\end{equation}
in dependence of $\lambda$ and $n$. The discussion at the beginning of this appendix yields the following characterizations.

\begin{theo}
\label{app-theor} 
\leavevmode
\begin{itemize}
\item[{\rm (a)}] A function $u=u(x)$ is a radially symmetric solution of Problem~\eqref{app-n-2}
if, and only if, $u$ satisfies~\eqref{app-repr} and $\sigma$ solves
\begin{equation*}
\sigma \, U_{n}^{\prime}(\sigma)=\gamma . \label{app-sigma}
\end{equation*}
\item[{\rm (b)}] A function $u=u(x)$ is a radially symmetric solution of Problem~\eqref{app-d-2}
if, and only if, $u$ satisfies~\eqref{app-repr} and $\sigma$ solves
\begin{equation*}
\sigma^{2}e^{U_{n}(\sigma)}=\lambda . \label{app-sigma-d}
\end{equation*}
\end{itemize}
\end{theo}

\begin{rem}
\label{inter}
In view of Theorem~\ref{app-theor}, counting solutions to the boundary value problems~\eqref{app-n-2} and~\eqref{app-d-2} amounts to counting intersections between the family of horizontal lines and the graphs of the functions $\sigma \mapsto \sigma \, U_{n}^{\prime}(\sigma)$ and $\sigma \mapsto \sigma^{2}e^{U_{n}(\sigma)}$, respectively.
\end{rem}

\subsection{Discussion}
For the Dirichlet problem~\eqref{app-d-2}, a complete description of the existence and multiplicity of solutions is available (see~\cite{JacSch02}). Specifically, we know that for any $n \ge 1$, there exists a finite positive value $\lambda_n^*$, such that there are no solutions when $\lambda>\lambda_n^*$.  
Moreover,
\begin{itemize}
\item for $n \in \{1,2\}$, there are two solutions for $\lambda<\lambda_n^*$ and exactly one solution when $\lambda=\lambda_n^*$;
\item for $3 \le n \le 9$, there are infinitely many solutions for $\lambda = \lambda_n^{**}:= 2(n-2)$ and a  finite but large number of solutions when $|\lambda- \lambda_n^{**}|$ is small;
\item for $n \ge 10$, there is one solution for $\lambda < \lambda_n^{*}=2(n-2)$ and no solution when $\lambda=\lambda_n^*$.
\end{itemize}
These results fit with the information that we obtain from Figure~\ref{fig2}, portraying
the function $\sigma \mapsto \sigma^{2} e^{U_{n}(\sigma)}$ for different values of $n$.
\begin{figure}[htb]
\centerline{
\includegraphics[width=.56\textwidth]{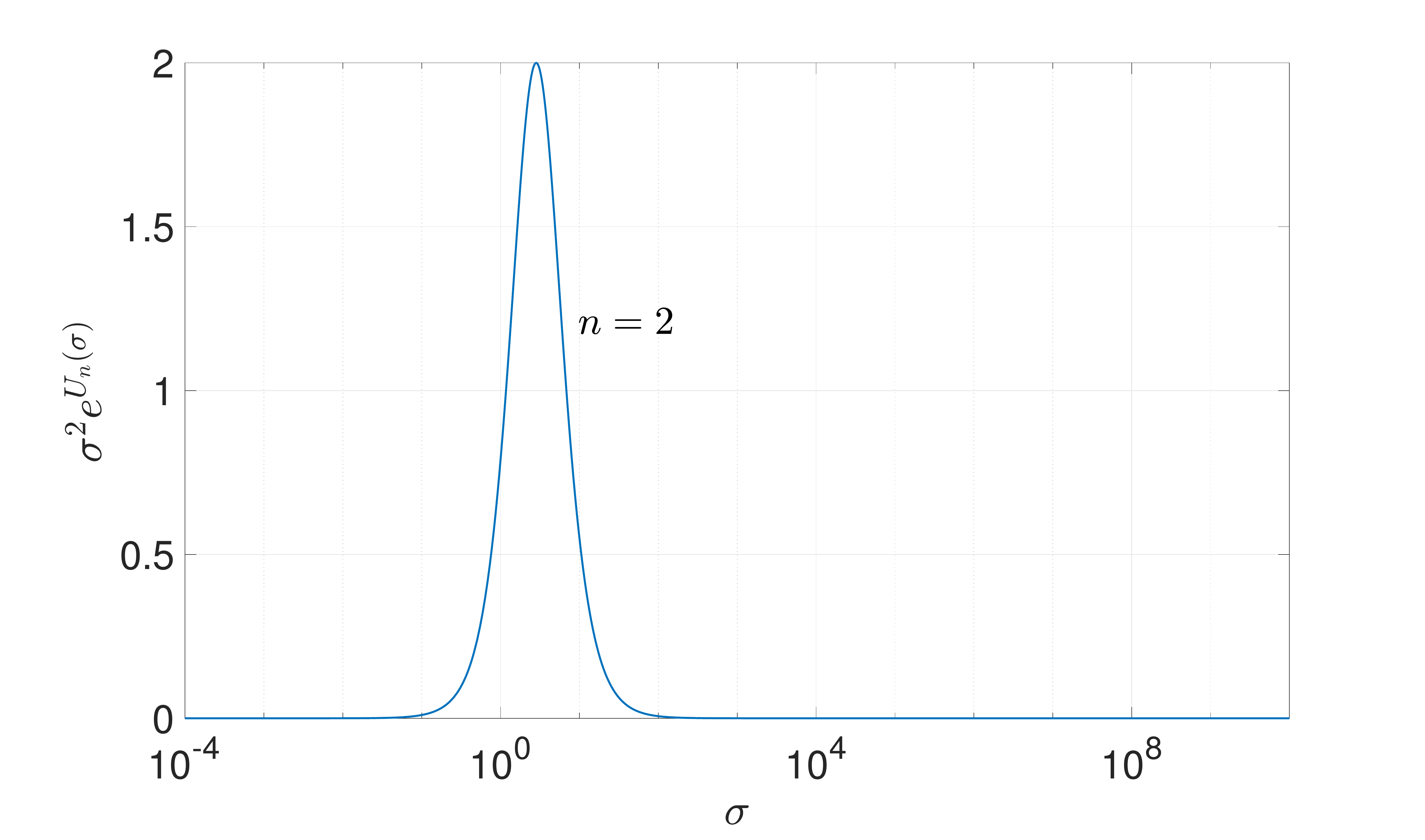} \hspace*{-.8cm}
\includegraphics[width=.56\textwidth]{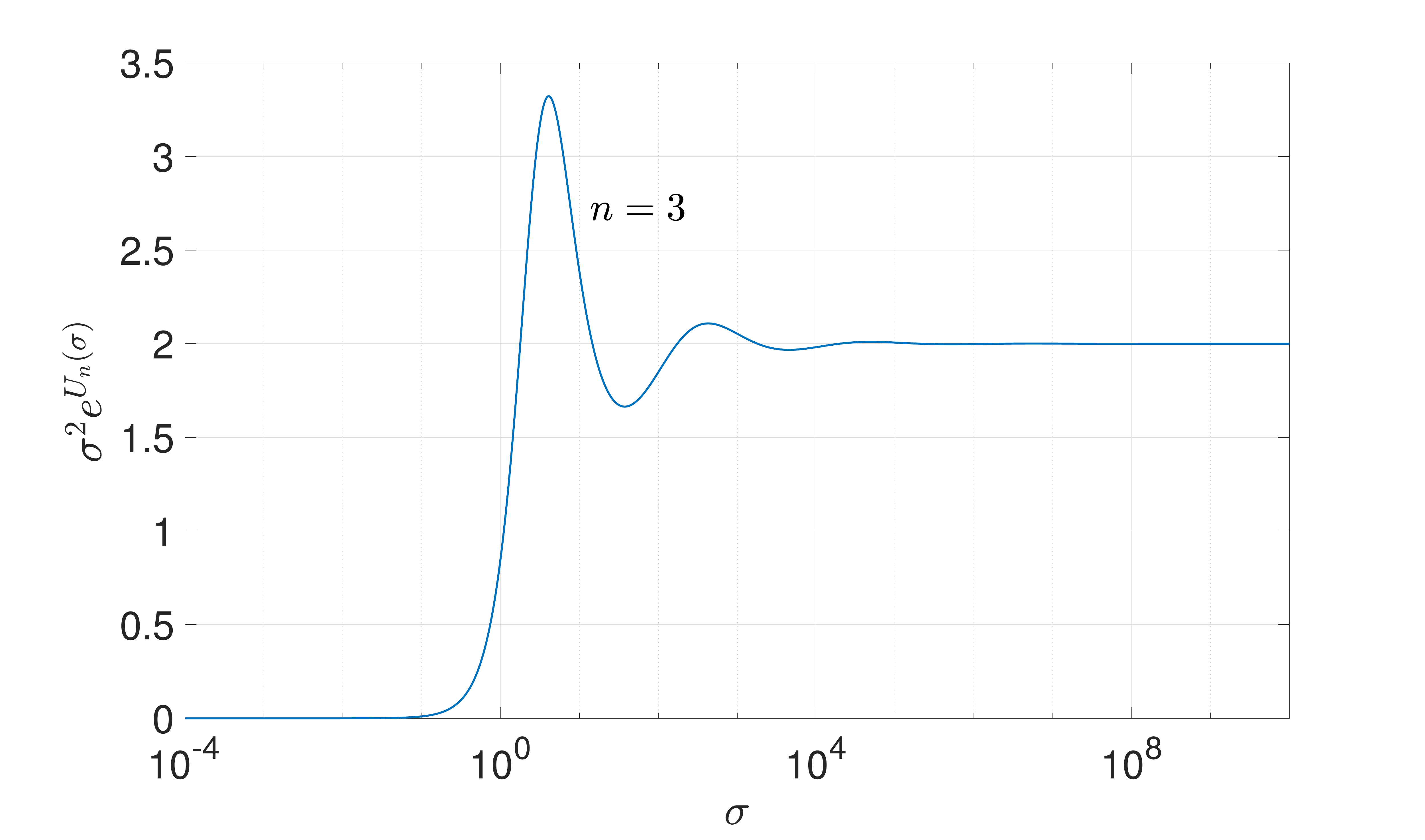} }
\centerline{
\includegraphics[width=.56\textwidth]{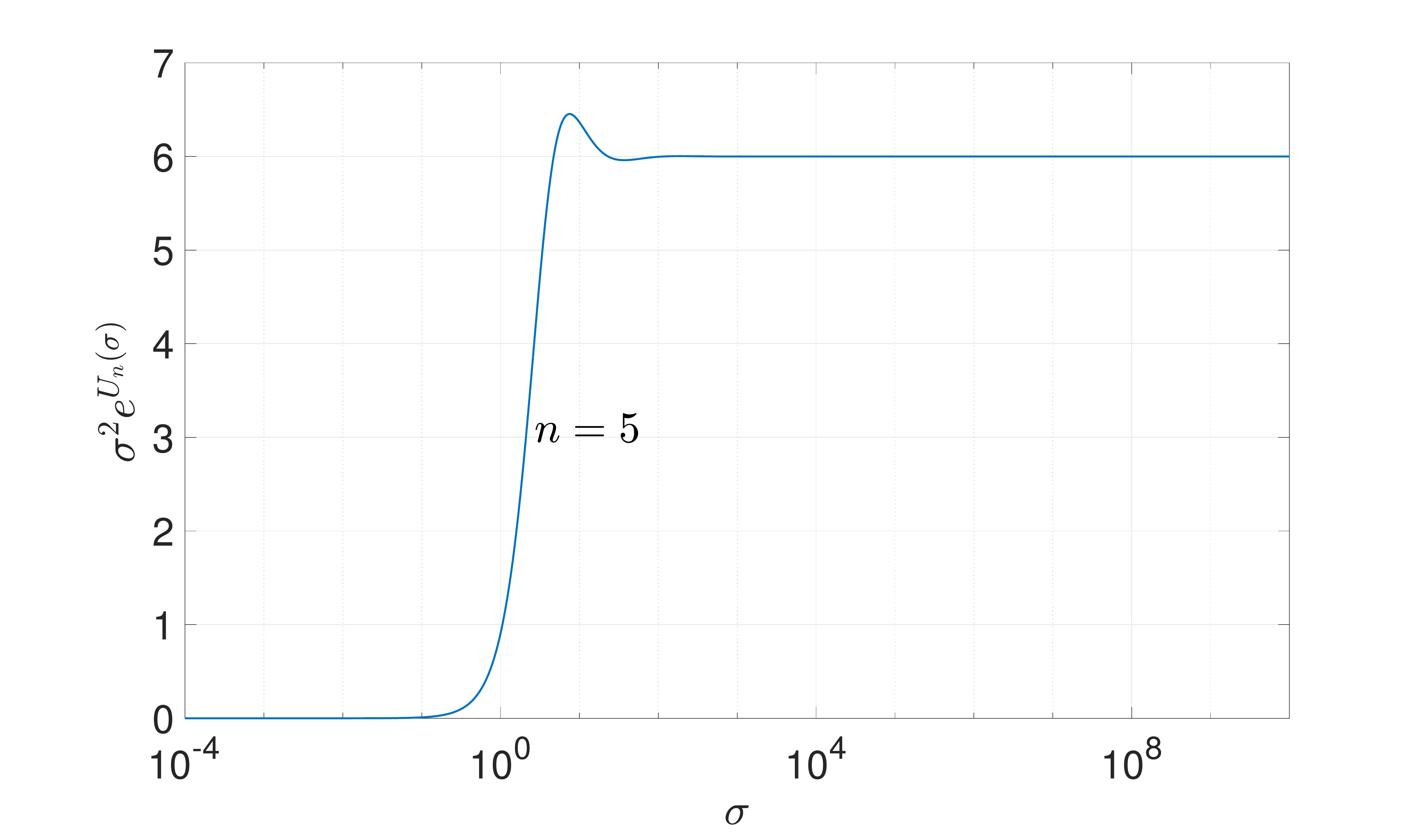} \hspace*{-.8cm}
\includegraphics[width=.56\textwidth]{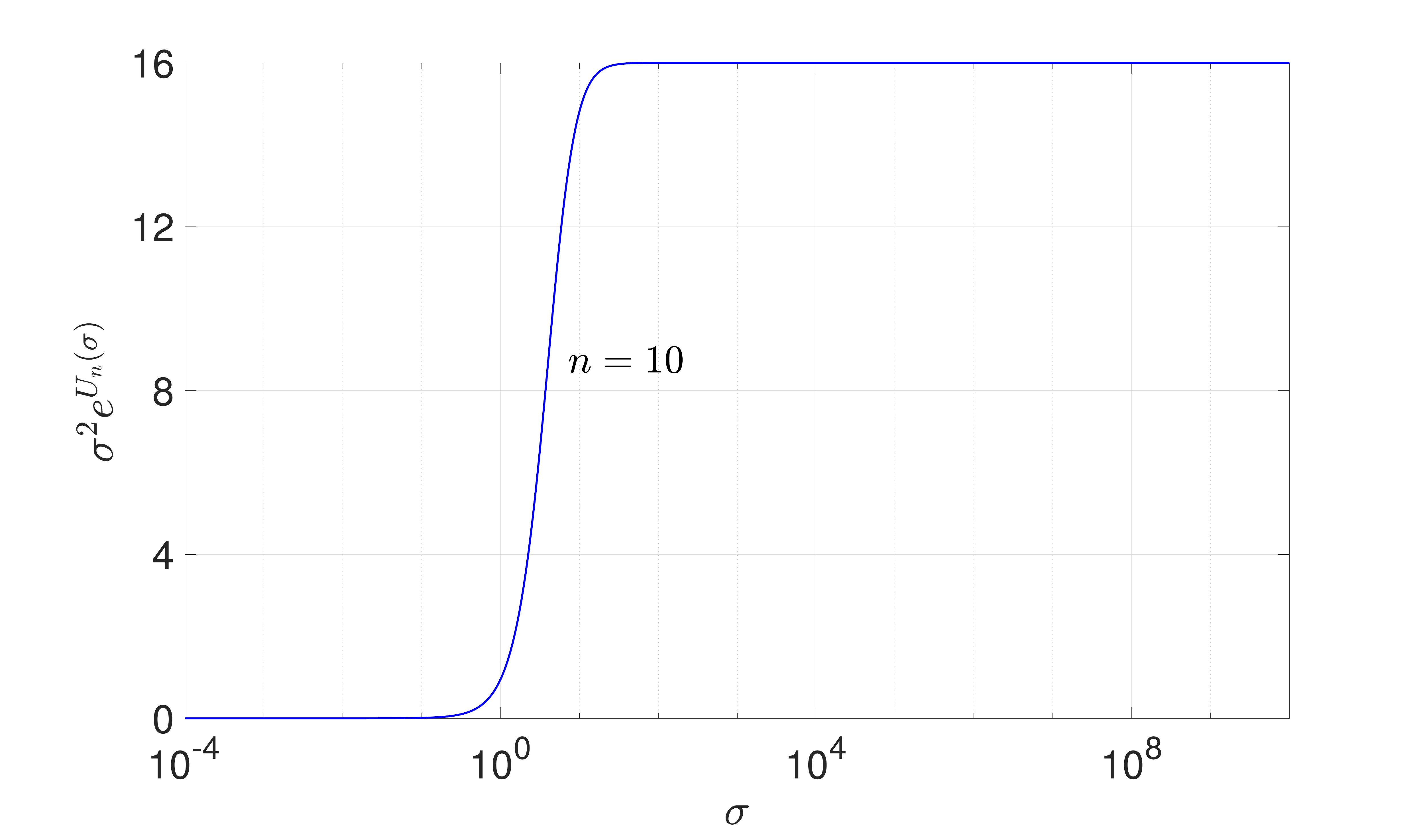} }
\caption{Profiles of the function $\sigma \mapsto \sigma^2 e^{U_{n}(\sigma)}$ for different values of $n$.}
\label{fig2}
\end{figure}
For the Neumann problem~\eqref{app-n-2} we are not aware of corresponding results.
Figure~\ref{fig3} portrays the function $\sigma \mapsto \sigma \, U_{n}^{\prime}(\sigma)$ for different values of $n$. 
\begin{figure}[htb]
\centerline{
\includegraphics[width=.56\textwidth]{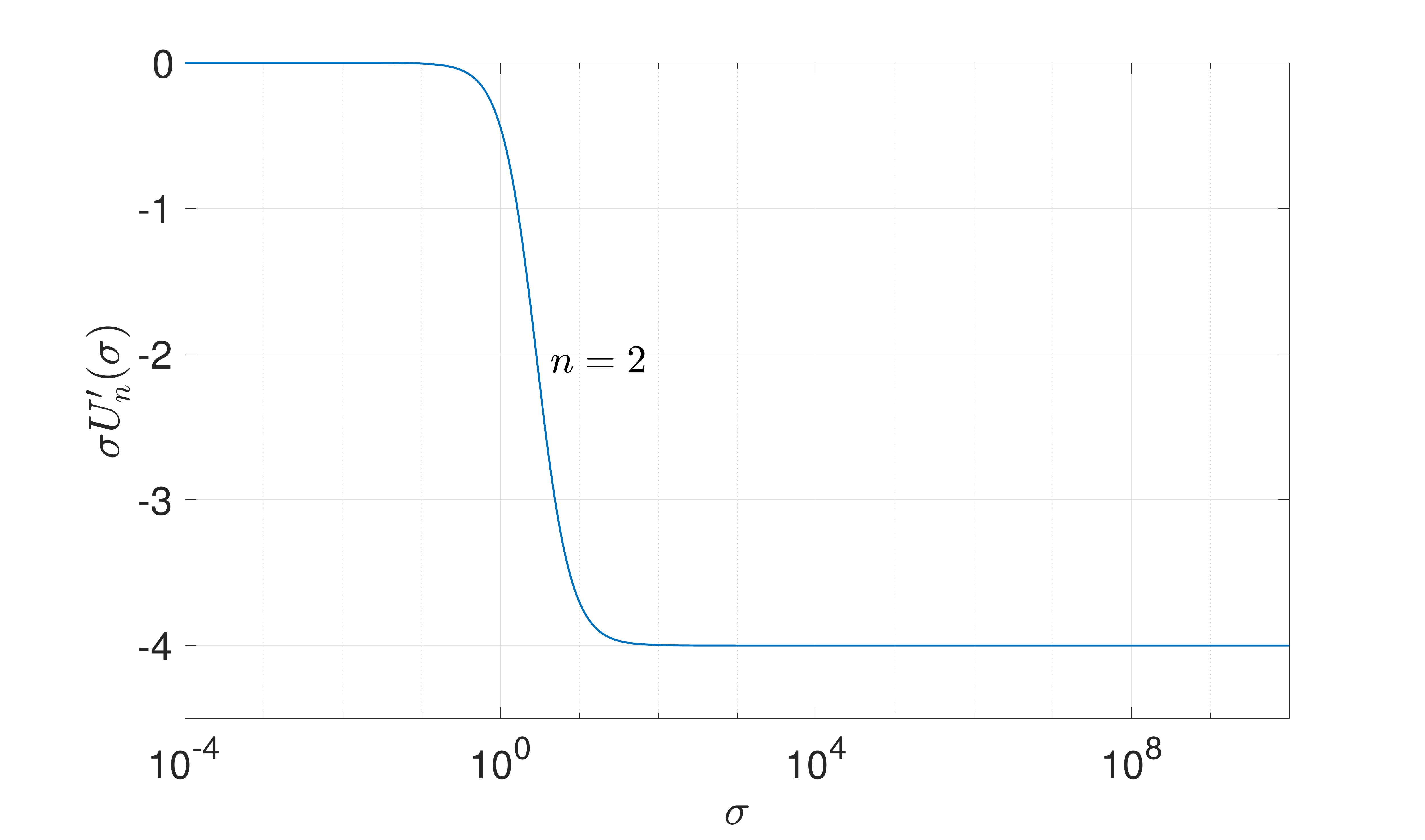} \hspace*{-.8cm}
\includegraphics[width=.56\textwidth]{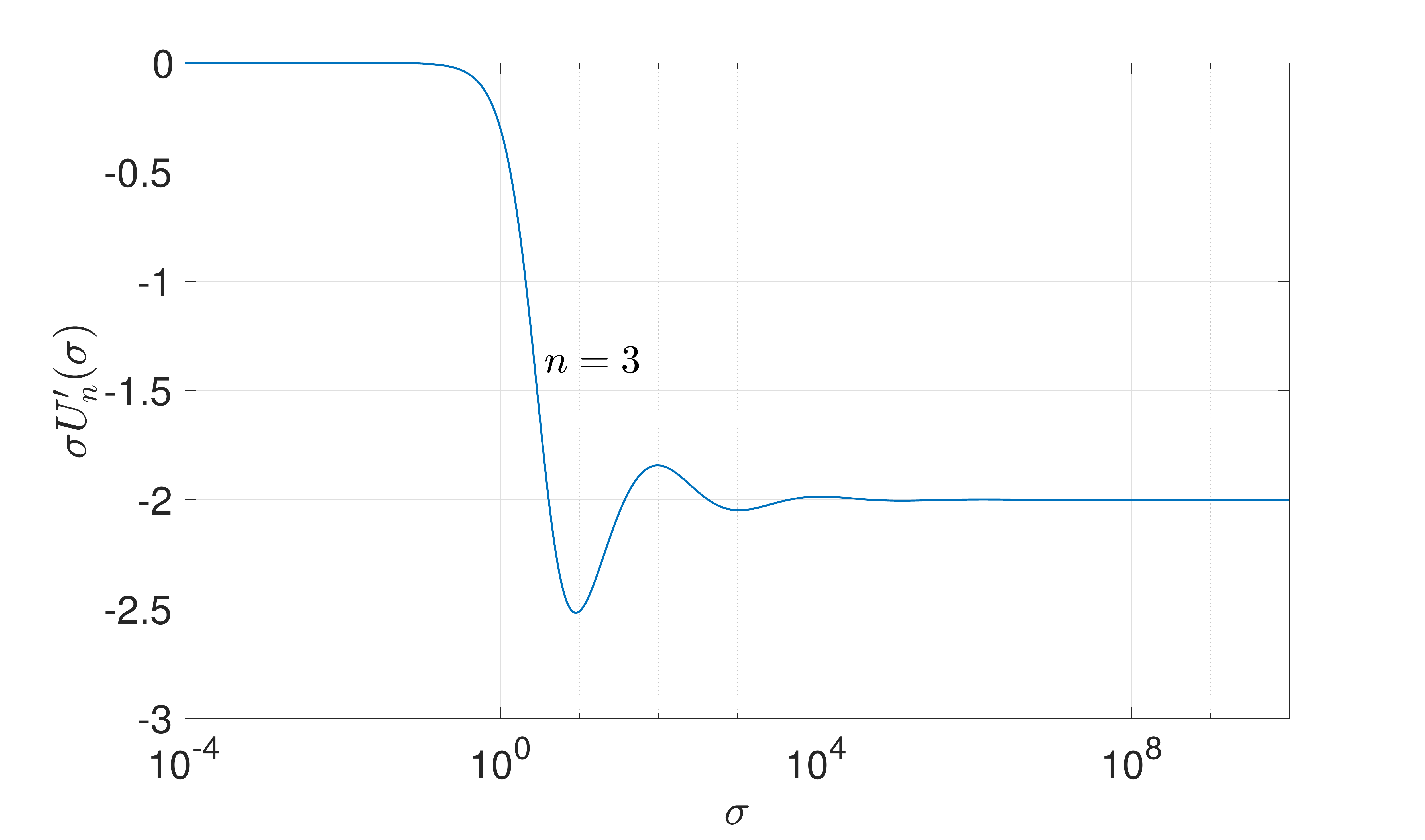} }
\centerline{
\includegraphics[width=.56\textwidth]{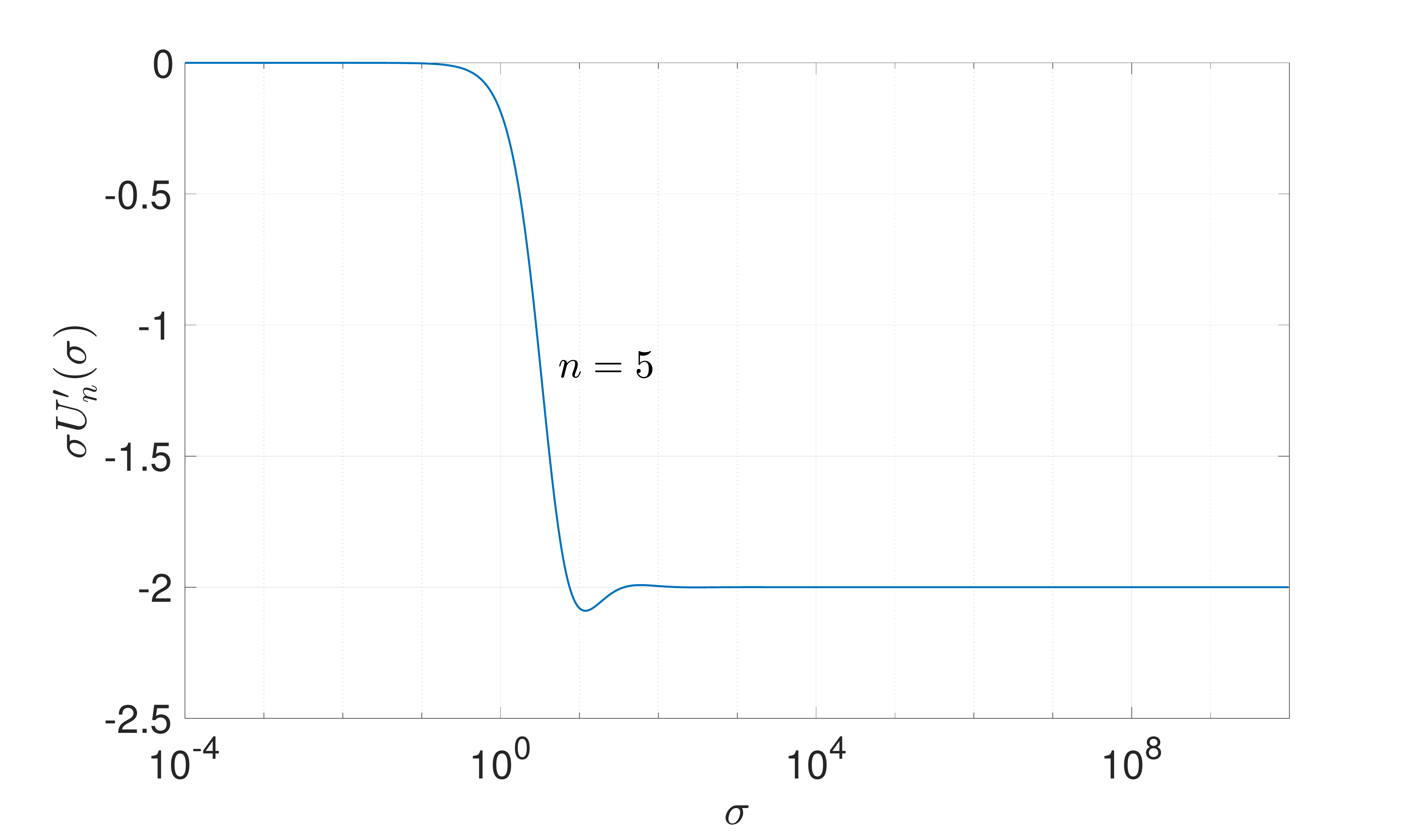} \hspace*{-.8cm}
\includegraphics[width=.56\textwidth]{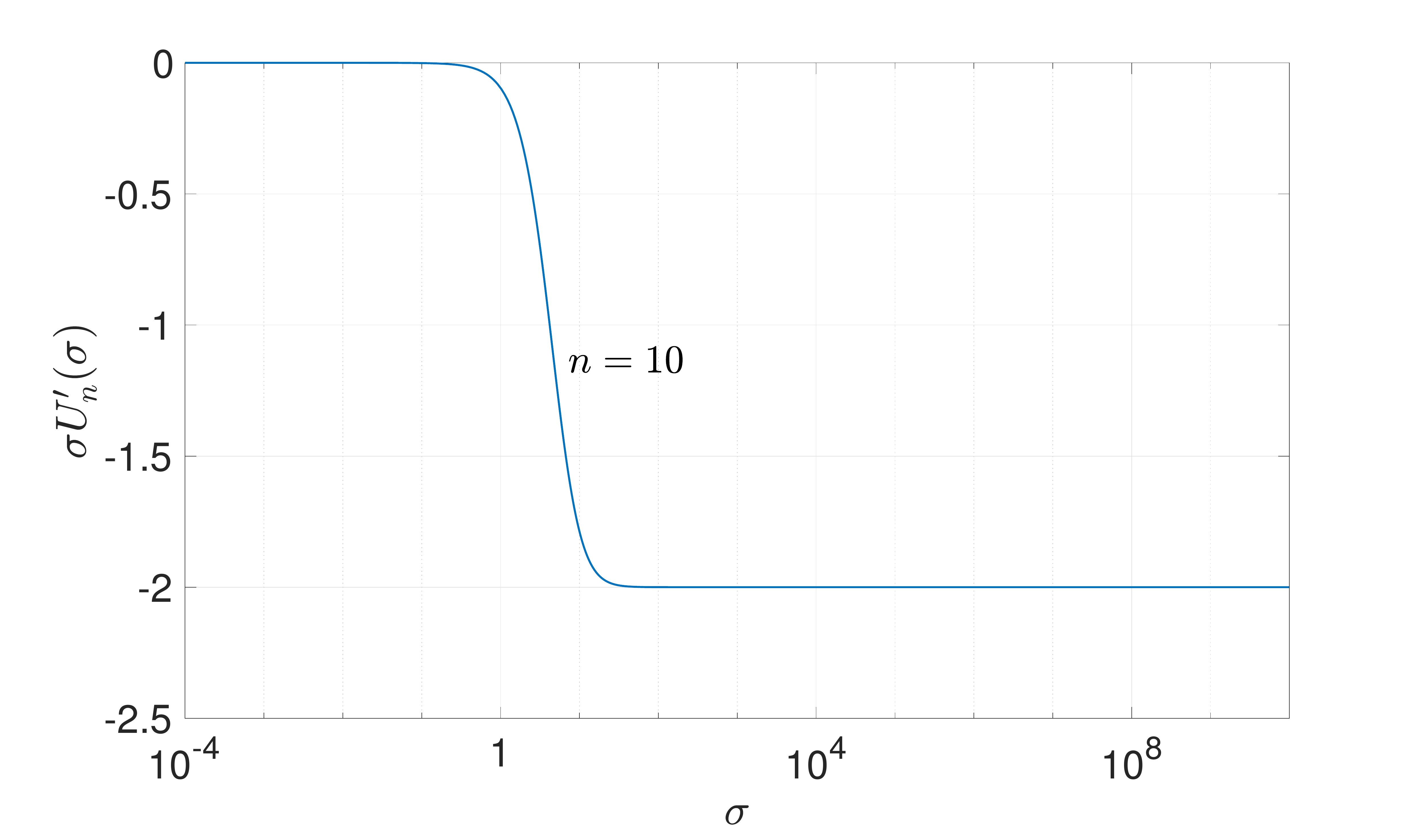} }
\caption{Profiles of the function $\sigma \mapsto \sigma U_{n}'(\sigma)$ for different values of $n$.}
\label{fig3}
\end{figure}
These graphs, in view of  Remark~\ref{inter}, lead to conclusions that depend on $n$, mirroring the different cases observed for the Dirichlet problem:

\medskip
\noindent for any $n \ge 1$, there exists a finite negative value $\gamma_n^*$, such that there are no solutions when $\gamma<\gamma_n^*$. Moreover,
\begin{itemize}
\item for $n =2$, there is one solution when $\gamma>\gamma_2^*=-4$;
\item for $n \in \{3,5\}$ (presumably, for $3 \le n \le 9$), there are infinitely many solutions for $\gamma = -2$ and a finite but large number of solutions when $|\gamma +2|$ is small;
\item for $n = 10$ (presumably, for $n \ge 10$), there is one solution when $\gamma>-2$.
\end{itemize}
Clearly, these assertions are based on numerical evidence and would require a proof.


\subsection*{Acknowledgments} 
This research study has been partially supported by the ESA's General Studies Programme under the contract ``Synergy between electromagnetic and gravitational fluid dynamics'' (no. 4000115042/15/nl/kml/fg).


\begin{thebibliography}{99}


\bibitem{ALSWjamc}
P.\,Amodio, T.\,Levitina, G.\,Settanni, E.B.\,Weinm\"uller.
On the Calculation of the Finite Hankel Transform Eigenfunctions. {\em J. Appl. Math. Comput.} \textbf{43}(1-2) (2013), 151--173.

\bibitem{ALSWcpc}
P.\,Amodio, T.\,Levitina, G.\,Settanni, and E.B.\,Weinm\"uller.
Numerical simulation of the whispering gallery modes in prolate spheroids. {\em Comput. Phys. Commun.} \textbf{185}(4) (2014), 1200--1206.

\bibitem{ASjnaiam}
P.\,Amodio and G.\,Settanni.
A matrix method for the solution of Sturm-Liouville problems.
{\em JNAIAM. J. Numer. Anal. Ind. Appl. Math.} \textbf{6}(1-2) (2011), 1--13.

\bibitem{ASaip}
P.\,Amodio and G.\,Settanni,
A stepsize variation strategy for the solution of regular Sturm-Liouville problems.
{\em AIP Conf.\ Proc.} \textbf{1389} (2011), 1335--1338.

\bibitem{ASjcam}
P.\,Amodio and G.\,Settanni,
A finite differences MATLAB code for the numerical solution of second order singular perturbation problems.
{\em J.\ Comput.\ Appl.\ Math.} \textbf{236}(16) (2012), 3869--3879.

\bibitem{ACMBVPTWP}
J.R.\,Cash, D.\,Hollevoet, F.\,Mazzia, A.M.\,Nagy.
Algorithm 927: The MATLAB code bvptwp.m for the numerical solution of two point boundary value problems.
{\em ACM Trans. Math. Software} {\bf 39}(2) (2013), art. 15.

\bibitem{CONDMESH2}
J.R.\,Cash, F.\,Mazzia. 
A new mesh selection algorithm, based on conditioning, for two-point boundary value codes.
{\em J.\ Comput.\ Appl.\ Math.} {\bf 184}(2) (2005),  362--381.

\bibitem{chandr}
S.\,Chandrasekhar. An Introduction to the Study of Stellar Structure. Dover
1957.

\bibitem{deficu}
P.\,Deuflhard,  B.\,Fiedler,  P.\,Kunkel. Efficient numerical pathfollowing beyond critical points. {\em SIAM J. Numer. Anal.} {\bf 24}(4) (1987),  912--927.

\bibitem{emden}
R.\,Emden. Gaskugeln. Teubner, Leipzig, Germany 1907.

\bibitem{gel}
I.M.\,Gelfand. Some problems in the theory of quasilinear equations. {\em Amer. Math. Soc. Transl.} {\bf 29} (1963), 295--381.

\bibitem{ESA19a}
D.\,Giordano, P.\,Amodio, F.\,Iavernaro, A.\,Labianca, M.\,Lazzo, F.\,Mazzia, L.\,Pisani. Fluid statics of a self-gravitating perfect-gas isothermal sphere. {\em Eur. J. Mech. B Fluids} {\bf 78} (2019), 62--87.




\bibitem{JacSch02} 
J.\,Jacobsen, K.\,Schmitt. The Liouville-Bratu-Gelfand problem for radial operators. {\em J.  Differential Equations} {\bf 184}(1) (2002), 283--298. 


\bibitem{jo-lun}
D.D.\,Joseph, T.S.\,Lundgren. Quasilinear Dirichlet problems driven by positive sources. {\em Arc. Rational Mech. Anal.} {\bf 49} (1972), 241-269.

\bibitem{TOM}
F.\,Mazzia, A.\,Sestini, D.\,Trigiante.
The continuous extension of the B-spline linear multistep methods for BVPs on non-uniform meshes. {\em  Appl. Numer. Math.} {\bf 59}(3-4) (2009),  723--738.

\bibitem{bvpSolve}
F.\,Mazzia, J.R.\,Cash, K.\,Soetaert.
Solving Boundary Value Problems in the open source software R: Package bvpSolve.
{\em Opuscula Math.} {\bf 34}(2) (2014),  387--403. 


\bibitem{CONDMESH1}
F.\,Mazzia, D.\,Trigiante. 
A hybrid mesh selection strategy based on conditioning for boundary value ODE problems.
{\em Numer. Algorithms} {\bf 36}(2) (2004),  169--187.

\end{thebibliography}
\end{document}